\documentclass[11pt]{article}
\usepackage{amsmath,amsthm,amssymb,amscd,fancybox,ifthen,epsfig,subfigure}
\usepackage[all]{xy}
\usepackage{times}

\newcommand\reg{\operatorname{reg}}

\newcommand\Dom{\operatorname{Dom}}

\newcommand\Span{\operatorname{span}}

\newcommand\cM{\mathcal M}

\newcommand\cQ{\mathcal Q}

\newcommand\hq{\mathfrak q}

\newcommand\tg{\widetilde{g}}
\newcommand\tb{\widetilde{b}}
\newcommand\tB{\widetilde{B}}
\newcommand\tpsi{\widetilde{\psi}}
\newcommand\tc{\widetilde{c}}
\newcommand\hc{\widehat{c}}

\newcommand\hf{\widehat{f}}

\newcommand\tL{\widetilde L}

\newcommand\tq{\widetilde{q}}
\newcommand\tu{\widetilde{u}}
\newcommand\tv{\widetilde{v}}
\newcommand\tQ{\widetilde{Q}}

\newcommand\RR{\mathbb R}

\newcommand\cC{\mathcal{C}}
\newcommand\cdC{\dot{\mathcal{C}}}

\newcommand\tf{\tilde{f}}

\renewcommand\Re{\operatorname{Re}}
\renewcommand\Im{\operatorname{Im}}

\newcommand\bbN{\mathbb N}

\newcommand\bbR{\mathbb R}

\newcommand\pa{\partial}

\newcommand\restrictedto{\upharpoonright}

\newcommand\supp{\operatorname{supp}}

\newcommand\CI{{\mathcal C}^{\infty}}

\newcommand\CIc{{\mathcal C}^{\infty}_{\text{c}}}

\newcommand\Id{\operatorname{Id}}

\DeclareMathOperator{\WF}{WF}

\newcommand{\e}{\epsilon}
\newcommand{\del}{\partial}

\newcommand{\calC}{{\mathcal C}}

\newcommand{\calO}{{\mathcal O}}

\newtheorem{theorem}{Theorem}
\newtheorem{proposition}{Proposition}
\newtheorem{corollary}{Corollary}
\newtheorem{lemma}{Lemma}

\theoremstyle{definition}
\newtheorem{definition}{Definition}

\theoremstyle{remark}
\newtheorem{remark}{Remark}

\begin{document}

\title{Wright-Fisher Diffusion in One Dimension}

\author{Charles L. Epstein,\footnote{Research partially supported by
    NSF grant DMS06-03973, DARPA HR00110510057, and the Thomas A. Scott Chair.
\quad
Address: Department of Mathematics, University of
    Pennsylvania; e-mail: cle@math.upenn.edu }\quad
and Rafe Mazzeo\footnote{Research partially supported by
    NSF grant DMS0805529, and  DARPA HR00110510057.\quad
Address: Department of Mathematics, Stanford
    University; e-mail: mazzeo@math.stanford.edu}}

\date{July 22, 2009}

\maketitle
\medskip
\noindent
{\sc{ This paper is dedicated to the memory of Ralph S.\ Phillips (1913-1998), one of the seminal
figures in the theory of semi-groups. Early in our  careers, he was a profound and positive influence on both of
us.}}
\medskip
\begin{abstract}  
We analyze the diffusion process associated to equations of Wright-Fisher
type in one spatial dimension. These are associated to the degenerate heat equation
\begin{equation}
\pa_t u=a(x)\pa_x^2u+b(x)\pa_xu
\end{equation}
on the interval $[0,1]$, where $a(x)>0$ on the interior and vanishes simply at
the endpoints, and $b(x)\pa_x$ is a vector field which is inward-pointing at
both ends. We consider various aspects of this problem, motivated by their
applications in biology, including a comparison of the natural boundary
conditions from the probabilistic and analytic points of view, a sharp
regularity theory for the ``zero flux'' boundary conditions, as well as a
derivation of the precise asymptotics of solutions of this equation, both as $t
\to 0, \infty$ and as $x \to 0, 1$. This is a precursor to our more complicated
analysis of these same questions for Wright-Fisher type problems in higher
dimensions.
\end{abstract}

\section{Introduction}
Consider the differential operator 
\begin{equation}
L = a(x)\frac{d^2\,}{dx^2} + b(x)\frac{d\,}{dx}
\label{eq:genwf}
\end{equation}
on the interval $[A,B]$, where the coefficient functions $a(x), b(x) \in
{\mathcal C}^\infty([A,B])$.  Our main assumptions on the coefficients are that
\begin{equation}
\begin{split}
a(x) & =  (x-A)(B-x)\tilde{a}(x), \\   & \mbox{where}\quad 
\tilde{a}(x) \in {\mathcal C}^\infty([A,B]), 
\mbox{and}\quad \tilde{a}(x) > 0 \ \forall\, x \in [A,B]
\end{split}
\label{eq:wfhyp1}
\end{equation}
and
\begin{equation}
b(A) \geq 0, \quad b(B) \leq 0.
\label{eq:wfhyp2}
\end{equation}

In other words, we assume that $a(x)$ vanishes simply at the endpoints of the
interval and nowhere else, and that the first order term $b(x)d/dx$ is an
inward pointing vector field. The diffusion associated to $L$ is of importance
in population biology. The basic example is the so-called Wright-Fisher
operator
\begin{equation}
L_{\mathrm{WF}} = x(1-x)\frac{d^2\,}{dx^2}
\end{equation}
on the interval $[0,1],$  which is the diffusion limit of a Markov chain
modeling the frequency of a gene with 2 alleles, without mutation or
selection. If the mutation rates between the alleles are $\mu_{12}$ and
$\mu_{21},$ and the mutant allele has selective advantage $s,$ then
\begin{equation}
  b(x)=\mu_{12}(1-x)-\mu_{21}x+sx(1-x),
\end{equation}
see~\cite{Ewens}.  Accordingly, we shall call any operator of the form
(\ref{eq:genwf}) with coefficients satisfying (\ref{eq:wfhyp1}) and
(\ref{eq:wfhyp2}) a generalized Wright-Fisher operator.  By an affine
transformation, we can also reduce to the case where $L$ is defined on the
interval $[0,1]$.

In a seminal 1952 paper,~\cite{Feller1}, Feller considered the additional
conditions one should impose on the domain of an operator like $L$ to obtain
the generator of a positivity preserving $\cC^0$-semi-group. Feller's starting
point was the Hille-Yosida theorem and his analysis centered on the
construction of the resolvent kernel, $(L-\lambda)^{-1}.$ Our approach, by
contrast, focuses directly on the Schwartz kernel for $e^{tL}$.  In addition,
we restrict to a particular choice of boundary condition, which arises in
applications to population genetics.

The plan of this paper is as follows. In the next section we recall a change of
variables introduced in~\cite{Feller1} which reduces a general Wright-Fisher
operator $L$ to one with principal part $x(1-x)\pa_x^2$.  We then describe the
natural boundary conditions for this operator and its adjoint.
The operator $L$ is modeled near each endpoints, $x=0$ or $x=1$ by an operator
of the form:
\begin{equation}
L_b=x\pa_x^2+b\pa_x, 
\end{equation}
where $b$ is a nonnegative constant. The model operators act on functions on
$[0,\infty),$ with boundary conditions at $x=0$ induced from those for $L.$ The
next several sections are devoted to a careful analysis of the solution
operator for $\pa_t-L_b$.  After a discussion in \S~\ref{modmxprn} of maximum
principles in this setting, we derive in \S~\ref{s.modfs} an explicit formula
for the Schwartz kernel $k_t^b$ of this solution operator, and then in
\S~\ref{s.mapping} combine these ingredients to prove sharp mapping properties
for the kernels $k_t^b$.  The long \S~\ref{s.pests} contains a variety of
technical estimates needed to analyze perturbations of these model
operators. After these extensive preliminaries, it is straightforward to
assemble this information and express the fundamental solution for a general
Wright-Fisher operator $L$ on $[0,1]$ as a convergent Neumann series. This is
done in \S~\ref{htker4L}; this leads directly to the description of precise
asymptotics of solutions of $(\pa_t - L)u = 0$ as $t \to 0$ and as $x \to 0,
1$. An asymptotic expansion of the heat kernel for $L$ was obtained using
other methods by Keller and Tier in~\cite{KellerTier}.  The next section
\S~\ref{s.infgen} uses this to provide a characterization of the infinitesimal
generator of the induced semi-group on $\calC^0$ and its spectrum, which gives
the long-time asymptotics for solutions of the diffusion problem. In
\S~\ref{s.resolv}, we then use the Hille-Yosida theorem to study the higher
order regularity of solutions to the ``elliptic''equation $(\lambda-L)w=f,$ for
$f\in\cC^m([0,1]).$ Finally, in \S~\ref{adjsmgrp} we characterize the adjoint
semi-group, using various abstract results from the theory of semi-groups as
well as the specific analytic information accrued to this point. This allows us
to discuss the ``forward'' Kolmogorov equation, which is crucial for
applications in population genetics.

Our focus throughout is on the $\calC^0$ (and $\calC^m$) theory, rather than the (simpler)
$L^2$ theory. We shall return to a comparison between the $\calC^0$ and $L^2$ semi-groups
elsewhere. 

\begin{remark}[Notational Remark] We let $I\subset\bbR$ be an interval or a
  ray. We use the notation $\cC^m_b(I)$ for the space of functions with $m$
  continuous, bounded derivatives on $I;$ the notation $\cC^m_c(I)$ for the
  space of functions with $m$ continuous derivatives, and compact support on
  $I.$ Finally we use $\dot{\cC}^m([a,b))$ to denote the space of functions with $m$
  continuous derivatives on $[0,b),$ tending to zero at $b.$
\end{remark}

\centerline{\bf Acknowledgements} 

\noindent
We would like to thank Charlie Fefferman for showing us his construction of the
heat kernel for the model problem on the half line, Dan Stroock for sharing his
probabilistic approach to these questions, which appears in a recently
completed preprint written jointly with Linan Chen,~\cite{ChenStroock}, and
Nick Patterson for very helpful discussions on this subject. Finally, we are
both grateful to Ben Mann and the DARPA FunBio program for giving providing the
intellectual stimulus and opportunity to undertake this research.

\section{A change of variables}
The first task is to show that by a judicious change of variables, already
known to Feller \cite{Feller1}, any operator of the form (\ref{eq:genwf}) can be
reduced to a first order perturbation of the exact Wright-Fisher operator
$L_{\mathrm{WF}}$.

Given the operator (\ref{eq:genwf}), assume (by rescaling if necessary) that
\[
\int\limits_0^1\frac{ds}{\sqrt{a(s)}}=\pi.
\]
For the heat equation this amounts to rescaling the time variable.
We define a diffeomorphism of $[0,1]$ by setting:
\[
\xi(x)=\sin^2\left(\frac{\eta(x)}{2}\right), \qquad \mbox{where}\qquad
\eta(x)= \int\limits_0^x\frac{ds}{\sqrt{a(s)}}.
\]
 A short computation shows that $L$ becomes
\begin{equation}
  \xi (1- \xi) \frac{d^2\,}{d\xi^2} + \tilde{b}(\xi)\frac{d\,}{d\xi},
\end{equation}
where $\tilde{b}(\xi)$  has the same properties as the original function
$b(x)$, i.e.\ it is smooth on the closed interval, and $\tilde{b}(0) = b(0)$,
$\tilde{b}(1) = b(1)$, so in particular is inward pointing at the two boundary
points, and vanishes at a boundary point if and only if $b$ does.

Henceforth we return to using $x$ as the independent variable and $a$ and $b$
as the coefficient functions. We assume that our general Wright-Fisher operator
has the form
\begin{equation}
L = x(1-x)\pa_x^2 + b(x)\pa_x,
\label{eq:genwfr}
\end{equation}
where $b(0), -b(1) \geq 0$. As noted earlier, the case of principal interest in
population genetics is when these values lie in $[0,1),$ the so called ``weak
mutation'' regime; however, it is no more difficult to treat the general case.
In our analysis it is actually necessary to
consider larger values of these constants, to obtain higher order estimate for the cases
when $|b(x)|$  lies in $[0,1).$

\section{Natural boundary problems}\label{natbvp}
Because $L$ is degenerate at the boundary of the interval $[0,1]$, we must
carefully specify how boundary conditions for $L$ are formulated. The main
observation is that operator $\tilde{L} := x(1-x)L$ has regular (Fuchsian)
singularities at the two boundary points $x=0$ and $x=1$, hence may be studied
by standard ODE methods.  This is something of a red herring, as it is the
resolvent operator $(\lambda-L)^{-1}$ that governs the behavior of the heat
kernel; the resolvents of $L$ and $\tL$ are fundamentally different.  The
simple vanishing of the coefficient of the second order term in $L$ plays a
crucial role in the analytic properties of the solutions to the associated heat
equation and their applications in stochastic processes and population
genetics. In particular, the diffusion processes associated to $L$ are
qualitatively quite different from those associated to $\tilde{L}$.

We define the notion of indicial roots for $L$. The complex number $s$
is called an indicial root for this operator at $x=0$ if
\[
L x^s = \calO(x^{s}).
\]
Note that for any $s$, $Lx^s = x^{s-1} c(s,x)$ where $c$ is smooth up
to the boundary, so $s$ is an indicial root only if some leading order
cancellation takes place. Indeed,
\[
Lx^s = (s(s-1) + b(0)s)x^{s-1} + \calO(x^{s})
\]
and this leading coefficient vanishes precisely when
\[
s (s + (b(0) - 1))  = 0 \qquad \Longleftrightarrow \qquad s = 0, 1 - b(0). 
\]
There are some subtle differences in the analysis of $L$ when $b(0) < 1$, $b(0)
= 1$ or $b(0) > 1$.  The case of principal interest in biology is when $0 \leq
b(0) < 1$ (and similarly, $0 \geq b(1) > -1$). For the boundary conditions we
are considering it is no more difficult to handle the general case where $b(0)
\geq 0$, and, for technical reasons, it is actually very useful to do so. The
case $b(0) < 0$ behaves quite differently, mainly due to the loss of the
maximum principle.

It is well known that if $b(0) \in \RR^+$, $b(0) \notin \{0,1,2, \ldots\}$,
then any solution of $Lu = 0$ satisfies
\[
u = u_1(x) + x^{1-b(0)}u_2(x), \qquad \mbox{where} \qquad u_1, u_2 \in \calC^\infty([0,1)).
\]
We call the leading coefficients $u_1(0)$ and $u_2(0)$ the Dirichlet and
Neumann data of $u$, respectively. If $u$ solves the equation
formally, i.e.\ in the sense of Taylor series, we see that if $u_j(0) = 0,$ then
$u_j(x)$ vanishes to infinite order at $x=0$; in particular, if $u_2(0) = 0$,
then $u(x) \in \calC^\infty([0,1))$. When $b(0) = 0$, the expansion has the
slightly different form
\[
u = u_1(x) + x \log x \, u_2(x), \qquad u_1, u_2 \in \calC^\infty([0,1)),
\]
and we again consider $u_1(0)$ and $u_2(0)$ as the Dirichlet and Neumann data,
so $u \in \calC^\infty([0,1))$ if and only if $u_2(0) = 0$.  If $b(0) \in
{\mathbb N}$, then there is a regular solution $u_1,$ with $u_1(0)$
non-vanishing, and a singular solution, $u_2$ of the form
\begin{equation}
  u_2(x)=x^{1-b(0)}u_{2r}(x)+\log x\, u_{2l}(x),\quad u_{2r},\, u_{2l} \in \calC^\infty([0,1)),
\end{equation}
with both $u_{2r}(0)$ and $u_{2l}(0)$ non-zero. A solution is regular at zero
if and only if it is bounded as $x\to 0^+.$

There is an almost identical formal analysis of the asymptotics of solutions to $Lu =
f$ where $f \in \calC^\infty([0,1])$, and for each of these cases we have the
\begin{proposition}
  Let $u$ be a solution to $Lu = f$ where $b(0) \geq 0$, $b(1) \leq 0$, and $f
  \in \calC^\infty([0,1])$. Then $u$ is smooth up to $x=0$ if $u_2(0)=0$.
\end{proposition}

Later in this paper we we also consider the adjoint $L^t$ of the operator $L$, defined formally by 
\[
\int_0^1 (Lu) v\, dx = \int_0^1 u (L^t v)\, dx, \qquad u, v \in \calC^\infty((0,1)).
\]
Integration by parts gives
\begin{equation}
\begin{array}{rcl}
  L^tu & = & \frac{d^2\,}{dx^2}( x(1-x)u) - \frac{d\,}{dx}(b(x)u) = 
\frac{d\,}{dx} \left( \frac{d\,}{dx}((1-x)u) - b(x)u \right) \\[0.5ex]
  & = & x(1-x)\frac{d^2\,}{dx^2} u + [2(1-2x) - b(x)] \frac{d\,}{dx} u - [2 + b'(x)]u.
\end{array}
\label{eq:adjoint}
\end{equation}
For historical reasons, $\pa_t - L$ is called Kolmogorov's backwards equation,
while $\pa_t - L^t$ is called Kolmogorov's forward equation.
A short calculation shows that the indicial roots of $L^t$ at $x=0$ are
\begin{equation}
s = 0, b(0) - 1.
\label{eq:indrtsadj}
\end{equation}

Suppose that $u \in \calC^\infty([0,1]).$ If $b(0), -b(1) \notin\bbN\cup\{0\}$,
then we see that $\int_0^1 (Lu) v\, dx = \int_0^1 u (L^t v)\, dx$ holds if and
only if
\begin{equation}
\begin{array}{l}
\lim_{x\to 0^+}[\pa_x(xv(x))-b(0)v(x)]=0\text{ and } \\[0.5ex]
\lim_{x\to 1^-}[\pa_x((1-x)v(x))+b(1)v(x)]=0.
\end{array}
\label{eq:adjbc}
\end{equation}
These limits characterize the adjoint boundary condition.  Loosely speaking, a
function satisfying these conditions is of the form:
$x^{b(0)-1}(1-x)^{-(1+b(1))}\tv(x),$ with $\tv$ smooth at $0$ and $1.$ In this
connection, the boundary conditions for $L$ are often formulated as ``zero
flux'' conditions:
\begin{equation}
  \lim_{x\to 0^+}x^{b(0)}\pa_xu(x)=0,\quad
\lim_{x\to 1^-}(1-x)^{-b(1)}\pa_xu(x)=0.
\end{equation}
We shall return to a more precise discussion of the adjoint operator in
\S~\ref{adjsmgrp}.

\section{Reduction to model problems on the half-line}\label{modprob1}
The first step in the construction of the parametrix is to localize the
analysis to  neighborhoods of  the boundary points. This leads us
to study certain model problems, which are the focus of much of the rest of
the paper. Indeed, we establish a maximum principle, and also derive
explicit formul\ae\ for the fundamental solutions of these model operators, and
using these results, prove sharp mapping properties for the heat kernels
of these model problems. Finally, by perturbation and standard parametrix
methods, we prove analogous results for general Wright-Fisher operators.

Our goal is the construction of the ``heat kernel'' or solution operator for
the initial value problem for the diffusion equation:
\begin{equation}\label{genWFivp}
 \pa_tu- [x(1-x)\pa_x^2+b(x)\pa_x]u=0\text{ with }u(x,0)=f(x),
\end{equation}
where we recall that an explicit change of variable has reduced the coefficient to $\pa_x^2$
to this normal form. The function $b$ is assumed to be smooth on $[0,1],$ with
$b(0)$ and $-b(1)$ non-negative. As above, we let $L$ denote the differential operator
\[
Lu=\left[x(1-x)\pa_x^2+b(x)\pa_x\right] u.
\]
The additional requirements that
\begin{equation}\label{WFregcnds}
\begin{split}
u\in\cC^0([0,1]\times [0,\infty))&\cap\cC^1([0,1]\times (0,\infty))\text{ and, }\\
  \lim_{x\to 0^+,1^-}x(1-x)&\pa_x^2u(x,t)=0\text{ for }t>0,
\end{split}
\end{equation}
assure the uniqueness of the solution. This follows from a maximum principle
for super-solutions.
\begin{proposition} Let $u$ satisfy~\eqref{WFregcnds}, and $(\pa_t-L)u\leq 0$
in $[0,1]\times (0,\infty)$; for any $T>0$, let $D_T=[0,1]\times [0,T].$ Then
\[
\max_{\{(x,t)\in D_T\}}u(x,t)=\max_{\{x\in [0,1]\}}u(x,0).
\]
\end{proposition}
\begin{proof} Let $\epsilon>0,$ then $u_{\epsilon}=u+\epsilon(1+t)^{-1}$ is a
strict super-solution. The standard argument (see~\cite{John}) then shows that the maximum of
$u_{\epsilon}$ occurs along the distinguished boundary $\pa' D_T=\pa
D_T\setminus [0,1]\times\{T\}.$ The regularity hypotheses
in~\eqref{WFregcnds} show that the maximum cannot occur where $x=0$ or $1.$
For example, if the maximum occurred at $(0,t_0),$ then
$\pa_tu_\epsilon(0,t_0)=0,$ and $\pa_x u_{\epsilon}(0,t_0)\leq 0.$ The conditions $b(0)\geq 0,$ and
\[
\lim_{x\to 0^+}x(1-x)\pa_x^2u_{\epsilon}(x,t_0)=0,
\]
contradict the fact that $u_{\epsilon}$ is a strict super-solution. Thus
\[
\max_{\{(x,t)\in D_T\}}u(x,t) < \max_{\{(x,t)\in D_T\}}u_{\epsilon}(x,t)=\max_{\{x\in  [0,1]\}}u_{\epsilon}(x,0)<
\max_{\{x\in  [0,1]\}}u(x,0)+\epsilon.
\]
This estimate holds for any $\epsilon>0,$ which completes the proof of the proposition.
\end{proof}

The transformation
\[
y=\sin^2\sqrt{x}
\]
maps the interval $[0,(\pi/2)^2]$ bijectively onto $[0,1],$ is regular at $x=0,$ but becomes singular at $(\pi/2)^2.$ 
In the $x$-variable, the operator $L_{\WF}=y(1-y)\pa_y^2$ transforms to
\begin{equation}\label{WFnrmfrm}
L_{\WF}=x\pa_x^2+x\hc(x)\pa_x.
\end{equation}
Here $\hc(x)$ is a real analytic function in a neighborhood of $x=0,$ which includes $[0,(\pi/2)^2),$ with 
$\hc(0)=-1/3.$ More generally, $L$ is carried to 
\begin{equation}\label{WFnrmfrm2}
 L=x\pa_x^2+x\hc(x)\pa_x+\tc(x)\pa_x,
\end{equation}
where $\tc$ is again a smooth function which satisfies
\[
\tc(0)=b(0).
\]
The coefficient $\tb$ of the first order term has the form
\[
\tb(x)=b_0+xc(x),\text{ where }b_0=b(0).
\]
This suggests introducing the half-line model operators
\[
L_b=x\pa_x^2+b\pa_x.
\]
In this coordinate the operator $L$ is represented in the form
\[
L=L_{b_0}+xc(x)\pa_x.
\]
The first order term $xc(x)\pa_x$ is, in a precise sense, a lower order
perturbation of $L_{b_0}.$ We show, in the succeeding sections, how to
construct the solution operator for~\eqref{genWFivp} with regularity
conditions~\eqref{WFregcnds}, from the solution operators for $L_{b}.$ 
The next several sections are devoted to the analysis of these solution
operators.

\section{Maximum principles and uniqueness for the model operators}\label{modmxprn}
For the heat equation on Euclidean space, the maximum principle and the
resulting uniqueness of solutions on $\bbR^n \times [0,\infty),$ satisfying
given initial conditions, requires that we impose a growth hypothesis on
solutions at $|x|\to\infty$. In this section we establish analogous results for
the model operators $L_b$, $b\geq 0$.

For any $T>0$, we consider solutions on the domains
\[
D_T=[0,\infty)\times [0,T].
\]
\begin{proposition}\label{maxprinciple} Suppose that $b\geq 0,$ and
\begin{equation}\label{maxprnreg}
u\in\cC^0(D_T)\cap\cC^1([0,\infty)\times (0,T])\cap \cC^2((0,\infty)\times (0,T])
\end{equation}
is a super-solution for the heat equation associated to $L_b$, so
$\pa_t u-L_bu\leq 0$. Suppose further that
\begin{equation}\label{2ndderest}
\lim_{x\to 0^+}x\pa_x^2u(x,t)=0,\text{ for }0<t\leq T,
\end{equation}
and that for some $a,M > 0$, $|u(x,t)|\leq Me^{ax}$ for all $(x,t) \in D_T$; then
\[
\sup_{(x,t)\in D_T}u(x,t)=\sup_{x\in [0,\infty)}u(x,0).
\]
\end{proposition}

\begin{proof} For $k$ a non-negative integer, define the function
\[
v_{\tau,k}(x,t)=\frac{1}{(\tau-t)^k}e^{\frac{x}{\tau-t}}=\pa_x^kv_{\tau,0}.
\]
Clearly $v_{\tau,k}(x,t)> 0,$ is smooth up to $x=0,$ and one readily checks that
$(\pa_t-L_k)v_{\tau,k}=0.$ We shall use the $v_{\tau,k}$ as barriers.

The proof is much the same as in the Euclidean case. We first treat the case
$b=k,$ a non-negative integer. Choose $\epsilon_1,\epsilon_2 > 0$ 
and $\tau \in (0,1/a)$, and consider the functions
\[
u_{\epsilon}(x,t)=u(x,t)-\epsilon_1 v_{\tau,b}(x,t)+\frac{\epsilon_2}{1+t}.
\]
These are defined in $D_{\tau}\cap D_T$ and satisfy
\begin{equation}\label{spsln1}
\pa_tu_{\epsilon}-L_bu_{\epsilon}<0
\end{equation}
there, hence are strict supersolutions. Our choice of $\tau$ ensures that $v_{\tau,b} \to
+\infty$ more rapidly than $e^{ax}$ as $x \to \infty$, so $u_\epsilon < 0$ for
$x$ sufficiently large, uniformly  for all $t \in [0,\tau]$; therefore, the
standard proof of the maximum principle shows that:
\begin{equation} \label{mxprn1}
u_{\epsilon}(x,t)\leq \sup_{(x,t)\in \pa^\prime (D_{\tau}\cap
  D_T)}u_{\epsilon}(x,t) \qquad \text{for}\quad  (x,t)\in D_{\tau}\cap D_T, 
\end{equation}
where $\pa^\prime D_T$ is the distinguished boundary $\left([0,\infty)\times
\{0\}\right)\cup \left(\{0\}\times [0,T)\right)$.  Clearly $\sup u_{\epsilon}$ is attained at some
point in $D_T$.

First observe that the maximum of $u_{\epsilon}(x,t)$ cannot occur along $x=0$ with $t>0.$ Indeed, 
suppose that the maximum occurs at $x=0,$ for some $t_0 \in (0,T)$.  If $b=0$, then the regularity 
of $u_{\epsilon}$ and \eqref{spsln1} give $\pa_tu_{\epsilon}(0,t)<0$, which contradicts that $u_\e$ 
reaches a maximum at $(0,t_0)$. On the other hand, if $b>0$, then~\eqref{spsln1} would imply that
\[
\pa_tu_{\epsilon}(0,t_0)=0\text{ and }\pa_xu_{\epsilon}(0,t_0)>0,
\]
which is also a contradiction since it shows that $u_{\epsilon}(0,t_0)<u_{\epsilon}(x,t_0)$ for $x$ small.  
If the maximum were to occur at $u(0,T)$, then $\pa_tu_{\epsilon}(0,T)\geq 0$, so as before,
$\pa_xu_{\epsilon}(0,T)>0$ and hence $u_{\epsilon}(0,T)<u_{\epsilon}(x,T),$
for small positive $x$.  Hence the maximum does not occur along $x=0.$ Note
that this argument relies strongly on the assumption of $\cC^1$ regularity up to $x=0$,
as well as~\eqref{2ndderest}.

All of this shows that we can replace~\eqref{mxprn1} by
\[
 u_{\epsilon}(x,t)\leq \sup_{x\in [0,\infty)}u_{\epsilon}(x,0)\leq
\sup_{x\in [0,\infty)}u(x,0) +\epsilon_2.
\]
Now letting $\epsilon_1,\epsilon_2\to 0,$ we conclude that
\begin{equation}\label{maxprn3}
\sup_{(x,t)\in D_{\tau}\cap D_T}u(x,t) \leq \sup_{x\in [0,\infty)}u(x,0).
\end{equation}
If $\tau\geq T,$ then this completes the proof of the proposition; otherwise 
repeat the argument recursively, with $u_1(x,t)=u(x,t+\tau)$ replacing
$u(x,t).$  Finitely many iterates produce the desired conclusion.

We now turn to non-integral values of $b.$ If $k\leq b<k+1,$ for a non-negative integer
$k,$ then 
\begin{multline*}
(\pa_t-L_b)\big( u-\epsilon_1v_{\tau,k+1}+\frac{\epsilon_2}{1+t}\big)=\\
(\pa_t-L_b)u+\epsilon_1(b-(k+1))v_{\tau,k+2}-\frac{\epsilon_2}{(1+t)^2} <0.
\end{multline*}
The inequality uses that $v_{\tau,k+2}> 0$. We can argue exactly as before with
\[
u_{\epsilon}(x,t)=u-\epsilon_1v_{\tau,k+1}+\frac{\epsilon_2}{t+1},
\]
to obtain~\eqref{maxprn3}, and iterate, if needed, to complete the proof of the proposition.
\end{proof}

The uniqueness of classical, ``tempered'' solutions is an immediate corollary:
\begin{corollary}
Let $u$ satisfy \eqref{maxprnreg}, \eqref{2ndderest}, and be a solution to $\pa_tu-L_bu=0$
for some constant $b\geq 0$.  If $u(x,0)=0$ and
\begin{equation}
|u(x,t)|\leq Me^{ax}\text{ for }(x,t)\in D_T
\label{eq:upperbound}
\end{equation}
for some $a,M > 0$, then $u\equiv 0$ in $D_T.$
\end{corollary}
\begin{remark} The discussion of indicial roots in Section~\ref{natbvp} shows
  that other solutions to the initial value problem for $\pa_t-L_b$, have an
  asymptotic expansion at $x=0$ with the term $x^{1-b}$, so even if the initial
  data is smooth, those solutions are not $\cC^1$ up to $x=0$ and do not
  satisfy~\eqref{2ndderest}. Following Feller, the local boundary condition
  that singles out the smooth solution is the zero flux condition:
\begin{equation}\label{noflux1}
\lim_{x\to 0^+}x^b\pa_xu(x,t)=0.
\end{equation}
\end{remark}

\section{Fundamental solutions for the model operator}\label{s.modfs}
Consider the heat equation
\begin{equation}
\pa_t u = L_b u, \qquad x \geq 0,\qquad u(0,x) = f(x)
\label{eq:mhe}
\end{equation}
with boundary condition at $x=0$ for $t > 0$ dictated by the demand that
$u(t,x) \in \calC^m$ up to $x=0$ for all $m \in {\mathbb N}$ provided the
initial condition $f$ also lies in $\calC^m$ (and has moderate growth as $x \to
\infty$). This is clearly needed for the solution operator
to~\eqref{eq:mhe} to define a semi-group on $\calC^m$.

The goal in this section is to find an explicit expression for the fundamental
solution of this problem.  Our strategy is as follows: We first study the
Fourier representation of the heat kernel for $L_0$; this turns out to be
somewhat simpler to analyze, and it suggests the general form of the
corresponding kernel when $b > 0$. It is useful to have both a Fourier
representation and the explicit Schwartz kernels of these operators, so we
include material describing both approaches.
\begin{remark} We would again like to thank Charlie Fefferman for showing us
  his construction of the kernel $k^0_t(x,y),$ essentially~\eqref{eq:hk0},
  on which all our subsequent development is based. This kernel also appears
  in~\cite{KellerTier}. 
\end{remark}

\subsection{Fundamental solution for $\pa_t - L_0$}
We begin by seeking a function $E^0(x,\xi,t),$ which is a solution to $\pa_t E^0
= L_0 E^0,$ with initial condition $E^0(x,\xi,0) = e^{ix\xi}$. A first
guess is that $E^0$ should involve a Gaussian, but because of the natural
appearance of the variable $\sqrt{x}$ in this problem (or, equivalently, the
fact that the dilation $(t,x) \mapsto (\lambda t, \lambda x)$ preserves the
equation under consideration), we are led to the ansatz $E^0(x,\xi,t) = \exp(x
\phi(\xi,t))$ for some as yet unknown function $\phi$. Inserting this into the
equation leads to the initial value problem
\[
\del_t \phi = \phi^2, \quad \phi(\xi,0) = ix\xi,
\]
which yields that
\begin{equation}
E^0(x,\xi,t) = \exp\left(\frac{-x}{t+i\xi^{-1}}\right) = \exp \left( \frac{ix\xi}{1 - it \xi } \right).
\label{eq:hkf0}
\end{equation}
As an oscillatory integral, the corresponding Schwartz kernel is 
\begin{equation}
k_t^0(x,y) = \frac{1}{2\pi} \int_{-\infty}^\infty E^0(x,\xi,t)e^{-iy\xi}\, d\xi.
\label{eq:hk0}
\end{equation}
With a bit of algebra, one sees that the integrand decreases exponentially as a function of $\mbox{Im}\,\xi$ 
when $x, t > 0$ and $y < 0$. Thus, a contour deformation shows that $k_t^0$ vanishes in this region. In
other words, $k_t^0(x,\cdot)$ is supported in $y \geq 0$ when $x,t > 0$. 

For the special case $b=0,$ the condition $u(0,t)=0$ defines a
$\cC^0_0$-semi-group, which turns out to be a sub-semi-group of  that defined
by $k_t^0.$ We call this the ``Dirichlet semi-group.'' Since $k_t^0(0,y) = \delta(y)$,
the kernel we have constructed is not the Dirichlet heat kernel, but can be
obtained by a simple modification:
\begin{equation}
k_t^{0,D}(x,y) = k_t^0(x,y) - e^{-x/t}\delta(y).
\label{eq:k0k0D}
\end{equation}
In other words, the fundamental solution for (\ref{eq:mhe}) with Dirichlet
boundary conditions at $x=0$ is
\begin{equation}
k_t^{0,D}(x,y) = \frac{1}{2\pi} \int_{-\infty}^\infty \left(e^{(ix\xi)/(1-it
\xi) - iy\xi} - e^{-x/t - iy\xi}\right)\, d\xi. 
\label{eq:hkd0}
\end{equation}
The integrand here again diverges like $1/\xi$ as $|\xi| \to \infty$, so we interpret
(\ref{eq:hkd0}) as an oscillatory integral via the regularization (obtained by
a formal integration by parts)
\begin{equation}\label{eqn183.1}
k_t^{0,D}(x,y)=-\left(\frac{x}{y}\right)\int\limits_{-\infty}^{\infty}
\frac{E(x,\xi;t)e^{-iy\xi}}{(t\xi+i)^2}\frac{d\xi}{2\pi}.
\end{equation}
This satisfies the same heat equation, but now vanishes $x=0$, so is indeed the
Dirichlet heat kernel.

It is possible to find an explicit expression for
$k_t^{0,D}$ in terms of elementary transcendental functions. First note that
$k_t^{0,D}$ is a function of
\[
\alpha =\frac{x}{t}\text{ and } \beta =\frac{y}{t}.
\]
Indeed, setting $\eta =\xi t,$ and $z = \eta + i$, we find that
\[
k_t^{0,D}(x,y)=-\left(\frac{\alpha}{t\beta}\right)\int\limits_{-\infty}^{\infty}
\frac{e^{\frac{-\alpha \eta }{\eta +i}}e^{-i\beta \eta }}{(\eta +i)^2}\frac{d\eta }{2\pi} = 
-e^{-(\alpha+\beta)}\left(\frac{\alpha}{t\beta}\right)\int\limits_{\Gamma_1}
\frac{e^{\frac{i\alpha}{z}-i\beta z}}{z^2}\frac{dz}{2\pi},
\]
where, by definition, $\Gamma_{\tau}=\{z:\: \Im z=\tau\}$. By an elementary contour deformation, 
\[
\int\limits_{\Gamma_{1}} \frac{e^{\frac{i\alpha}{z}-i\beta z}}{z^2}\frac{dz}{2\pi}= 
\int\limits_{\Gamma_{\tau}} \frac{e^{\frac{i\alpha}{z}-i\beta z}}{z^2}\frac{dz}{2\pi}-\int\limits_{|z|=1}
\frac{e^{\frac{i\alpha}{z}-i\beta z}}{z^2}\frac{dz}{2\pi}
\]
for any $\tau < 0$. The first term on the right vanishes since it can be made
arbitrarily small by letting $\tau \to -\infty$, so it suffices to compute the
integral on the unit circle. For this, note that the coefficient of $1/z$ in
\[
z^{-2}\exp(i\alpha/z) \times \exp(-i\beta z) = z^{-2}\sum_{j=0}^\infty \frac{1}{j !}(i\alpha/z)^j 
\sum_{k=0}^\infty \frac{1}{k !}(-i\beta z)^k
\]
is equal to
\[
\sum_{\ell=0}^\infty \frac{1}{\ell !\, (\ell+1) ! }(-i\beta )^{\ell+1}(i\alpha )^{\ell} = -i\beta 
\sum_{\ell=0}^\infty \frac{(\alpha \beta)^\ell}{\ell ! (\ell+1) !} = -i\beta I_1(2\alpha \beta),
\]
where $I_1$ is the modified Bessel function of order $1$. This yields, finally,
the explicit formula
\begin{equation}
k_t^{0,D}(x,y)=\frac{1}{t}e^{-\frac{x+y}{t}}\sqrt{\frac xy}I_1\left(\frac{2\sqrt{xy}}{t}\right)
\label{eq:shkd0}
\end{equation}
for the Dirichlet fundamental solution of the model heat equation. It is useful to
represent this kernel in an alternate form. For $b>0$ we define the entire functions
\[
\psi_b(z)=\sum_{j=0}^{\infty}\frac{z^j}{j!\Gamma(j+b)}.
\]
An elementary calculation shows that $\psi_b$ satisfies the ordinary
differential equation:
\begin{equation}\label{psiode}
z\psi_b''+b\psi_b'-\psi_b=0
\end{equation}
The Dirichlet heat kernel is then given by
\[
k_t^{0,D}(x,y)dy=\left(\frac{x}{t}\right)e^{-\frac{x+y}{t}}\psi_{2}
\left(\frac{xy}{t^2}\right)\frac{dy}{t}. 
\]
The kernel $k^0_t$ can be expressed as
\begin{equation}\label{k0decmp}
k^0_t(x,y)=e^{-\frac{x}{t}}\delta(y)+k^{0,D}_t(x,y).
\end{equation}

The relationship between $\psi_2$ and the $I$-Bessel function implies the
asymptotic expansion:
\[
\psi_2(z)\sim\frac{e^{2\sqrt{z}}}{\sqrt{4\pi}z^{\frac{3}{4}}} \left[1+\sum_{j=1}^\infty\frac{c_{2,j}}{z^{\frac j2}}\right].
\]
See~\cite[8.451.5]{GR}

\subsection{Fundamental solution for $\pa_t - L_b$}
We now undertake a similar analysis of the fundamental solution for the problem
(\ref{eq:mhe}) for $0 < b $. As explained in \S~\ref{natbvp} the boundary condition at
$x=0$ which should guarantee that solutions are smooth up to $x=0$ is the
analogue of the Neumann condition, i.e.\ the one which excludes the term
$x^{1-b}$. We denote by $k_t^b(x,y)$ the heat kernel for this problem with this
zero-flux conditions.

As before, we first determine the Fourier representation of $k_t^b$ using the ansatz that 
$E^b(t,x,\xi) = \psi(t\xi) e^{-x \phi(t,\xi)}$; this leads fairly directly to the expression
\begin{equation}
E^b(t,x,\xi) = (1-it\xi)^{-b} e^{\frac{ix\xi}{1-it\xi}},
\label{eq:hkfb}
\end{equation}
and hence, (when $b\leq 1$) as an oscillatory integral,
\begin{equation}
k_t^b(x,y) = \frac{1}{2\pi}\int_{-\infty}^\infty 
 \frac{e^{\frac{ix\xi}{1-it\xi} - iy\xi}}{(1-it\xi)^{b}}\, d\xi
\label{eq:hksfb}
\end{equation}
If $f\in\cC^0_c([0,\infty))$ has an absolutely integrable Fourier transform,
then this formula can be interpreted to mean that
\begin{equation}\label{fourepb}
\int\limits_0^{\infty}k_t^b(x,y)f(y)dy=\frac{1}{2\pi}\int\limits_{-\infty}^{\infty}
\frac{e^{\frac{ix\xi}{1-it\xi}}\hf(\xi)\, d\xi}{(1-it\xi)^{b}}.
\end{equation}

In order to evaluate \eqref{eq:hksfb}, we start off as before, setting $\alpha= x/t,$ 
$\beta=y/t,$ and expressing \eqref{eq:hksfb} as a contour integral over $\Im z=1:$
\[
k_t^b(x,y)=\frac{e^{\frac{\pi i b}{2}}}{2\pi t}
e^{-\frac{x+y}{t}}\lim_{R_1, R_2\to\infty}\int\limits_{-R_1+i}^{R_2+i}
e^{i\left(\frac{\alpha}{z}-\beta z\right)}z^{-b}\, dz.
\]
To define $z^{-b},$ we take $\arg z=0$ on the positive real axis and consider
$z^{-b}$ as defined on the plane cut along the negative imaginary axis.

Now change variables, setting $z=\sqrt{\alpha/\beta}\, \tau$, and for simplicity
write $\zeta=\sqrt{\alpha\beta}$, to write
\[
\int_{\Im z = 1} e^{i\left(\frac{\alpha}{z}-\beta z\right)}z^{-b}\, dz =
\sqrt{\frac{\alpha}{\beta}}
\int_{\Im \tau = \sqrt{\beta/\alpha}} e^{i\zeta\left(\frac{1}{\tau} - \tau\right)}\tau^{-b}\, d\tau.
\]
Since the integrand decays as $|\Re \tau| \to \infty$ and as $\Im \tau \to
-\infty$, we can deform the contour to yield 
\[
\Lambda=\{ye^{\frac{3\pi i}{2}}:y\in (\infty,1]\}\cup
\{e^{i\theta}:\:\theta \in [\frac{3\pi}{2},\frac{-\pi}{2}]\}\cup
\{ye^{\frac{-\pi i}{2}}:y\in [1,\infty)\}.
\]
Inserting the explicit parameterizations of the various parts of this contour
into the integrand
$e^{i\zeta\left(\frac{1}{\tau}-\tau\right)}\frac{d\tau}{\tau^b}$ and
simplifying yields the absolutely convergent representation:
\begin{multline}
  \int\limits_{\Lambda}
e^{\frac{\pi i b}{2}}e^{i\zeta\left(\frac{1}{\tau}-\tau\right)}\frac{d\tau}{\tau^b}=\\
\int\limits_{-\pi}^{\pi}e^{2\zeta\cos\phi}\cos(b-1)\phi d\phi-
\sin\pi(b-1)\int\limits_{1}^{\infty}e^{-\zeta(y+\frac{1}{y})}\frac{dy}{y^b}.
\notag
\end{multline}
Changing the variable to $y=e^t$ in the second integral on the right, we get
\[
 2\int\limits_{0}^{\pi}e^{2\zeta\cos\phi}\cos(b-1)\phi d\phi-
2\sin\pi(b-1)\int\limits_{0}^{\infty}e^{-2\zeta\cosh t}e^{-t(b-1)}dt.
\]
This combination of integrals appears as Formula 8.431.5 in~\cite{GR}, and is
seen to equal $2\pi I_{b-1}(2\zeta)$. Putting all of these calculations together
gives the two equivalent expressions
\begin{equation}
k_t^b(x,y)=\frac{1}{t}\left(\frac{x}{y}\right)^{\frac{1-b}{2}}e^{-\frac{x+y}{t}}
I_{b-1}\left(2\sqrt{\frac{xy}{t^2}}\right)
\label{eq:hkdb}
\end{equation}
and
\begin{equation}
k_t^b(x,y)\, dy =\left(\frac{y}{t}\right)^be^{-\frac{x+y}{t}}\psi_b\left(\frac{xy}{t^2}\right)\, 
\frac{dy}{y}.
\label{psifrm1}
\end{equation}
Using the classical asymptotics for $I_{b-1}$ we see that as $z \to \infty$, $\psi_b$ has 
the asymptotic expansion
\begin{equation}\label{psiexp1}
\psi_b(z)\sim\frac{z^{\frac14-\frac{b}2}e^{2\sqrt{z}}}{\sqrt{4\pi}}\left[1+
\sum_{j=1}^\infty\frac{c_{b,j}}{z^{\frac{j}2}}\right].
\end{equation}
From these explicit formul{\ae} it is evident that the kernels $k^b_t(x,y)$ are
pointwise positive. Once it is shown that they generate semi-groups, it 
follows that these semi-groups are positivity improving. The Fourier
representation has a very useful consequence:
\begin{lemma}\label{difnlem}
 If $f\in\cC^{m}_c([0,\infty))$ satisfies $\pa_x^j f(0)=0,$ for $1\leq j\leq  m,$  then
\[
\pa_x^j\int\limits_0^{\infty}k^b_t(x,y)f(y)\, dy= \int\limits_0^{\infty}k^{b+j}_t(x,y)\pa_y^jf(y)\, dy,
\text{ for }1\leq j\leq m.
\]
\end{lemma}
\begin{proof} This follows easily from~\eqref{fourepb}, as the hypothesis implies that we can differentiate under 
the integral sign. The vanishing conditions on $\pa_x^jf(0)$ give that $\xi^{m-1}\hf(\xi) \in L^1(\bbR)$, so that 
\[
\widehat{\pa_x^jf}(\xi)=(i\xi)^j\hf(\xi),\text{ for }1\leq j\leq m-1.
\]
If $\xi^m\hf(\xi) \in L^1(\bbR),$ then we can immediately extend the formula to
$j=m.$ In general, approximate $f$ in the $\cC^m$-norm by a sequence $\{f_n\}
\subset\cC^m,$ which satisfies $\pa_x^j f(0)=0,$ for $1\leq j\leq m$ and this
additional integrability hypothesis. A limiting argument then establishes the
truth of the lemma for $j=m$ in the stated generality.
\end{proof}

\begin{remark} If $0<b<1,$ then, as shown in~\cite{Feller1}, one can define the
  Dirichlet problem for $\pa_t-L_b.$ For this range of parameters this
  corresponds to having a non-zero coefficient of $x^{1-b},$ hence a non-smooth
  solution. For $b\geq 1,$ Feller showed that~\eqref{noflux1} is the unique
  local boundary condition that defines a positivity preserving semigroup.
\end{remark}

\subsection{Representation formula}
As a first application of the existence of these fundamental solutions, we prove a 
representation formula for the solution to the initial value problem
\begin{equation}\label{ivpb1}
\pa_t u=L_bu\text{ with }u(x,0)=f(x),
\end{equation}
where $f$ and $u$ satisfy (\ref{eq:upperbound}) and in addition,
\begin{equation}\label{bcb}
\lim_{x\to 0^+}x^b\pa_xu(x,t)=0\text{ for  }b>0,\quad
\lim_{x\to 0^+}x\pa_xu(x,t)=0\text{ if }b=0.
\end{equation}
Note that if $u$ is in $\cC^1(D_T),$ then it automatically satisfies
this boundary condition.
\begin{proposition}\label{prop3} Suppose that $u\in\cC^0(D_T)\cap\cC^2((0,\infty)\times (0,T))$
satisfies~\eqref{ivpb1},~\eqref{bcb}. If $f, u, \pa_t u$ and $\pa_xu$ satisfy
(\ref{eq:upperbound}) then for all $t<a^{-1}$, $u$ is given by the absolutely
convergent integral:
\[
u(x,t)=\int\limits_{0}^{\infty}k^b_t(x,y)f(y)dy.
\]
\end{proposition}
\begin{proof} For $b>0,$ the proof of this proposition is a simple integration by parts
  argument. It uses the fact that $k^b_t$ also satisfies the adjoint equation
  and boundary condition: if $t>0,$ then
\[
\pa_t k^b_t(x,y)=(\pa_y^2y-b\pa_y)k^b_t(x,y)\quad \text{and } \lim_{y\to
  0^+}[\pa_y yk^b_t(x,y)-bk^b_t(x,y)]=0.
\]
Moreover, it is not difficult to show that as $y\to\infty,$  with $x$ in a
compact set, we have 
\begin{equation}\label{kbnlrgy}
k^b_t(x,y)\leq Cy^{\frac{2b-3}{4}}e^{-\frac{y}{t}}.
\end{equation}

With these preliminaries, we can integrate the equation satisfied by $u$ to obtain
\[
0=\int\limits_{0}^{t-\epsilon}\int\limits_{0}^{\infty} k^b_{t-s}(x,y) \, (\pa_su(y,s)-L_bu(y,s)) \, dyds.
\]
Here $\epsilon$ is a small positive number. Provided $t<a^{-1},$ the estimates
satisfied by $u,\pa_t u,\pa_x u,$ and~\eqref{kbnlrgy} justify the subsequent
manipulations of this integral.

We first integrate by parts in $s$ to obtain:
\begin{multline*}
  \int\limits_0^{\infty}[u(t-\epsilon,y)k^b_{\epsilon}(x,y)- u(0,y)k^b_{t}(x,y)]\, dy=\\
  \int\limits_0^{\infty}\int\limits_0^{t-\epsilon}\left[L_bu(s,y)k^b_{t-s}(x,y)
    - u(s,y)\pa_tk^b_{t-s}(x,y)\right]\, dsdy.
\end{multline*}
Now integrate by parts in $y$ using the boundary conditions satisfied by $u$
and $k^b_{t-s}$ at $y=0,$ and the estimates~\eqref{eq:upperbound}
and~\eqref{kbnlrgy}, to conclude that since $t<a^{-1},$ then
\begin{multline*}
  \int\limits_0^{\infty}[u(t-\epsilon,y)k^b_{\epsilon}(x,y)- u(0,y)k^b_{t}(x,y)]dy=\\
  \int\limits_0^{\infty}\int\limits_0^{t-\epsilon}u(s,y)\left[L_b^tk^b_{t-s}(x,y)-
    \pa_tk^b_{t-s}(x,y)\right]dsdy.
\end{multline*}
For any $\epsilon>0,$ the right hand side of this equation vanishes identically, and therefore
\[
\int\limits_0^{\infty}u(t-\epsilon,y)k^b_{\epsilon}(x,y)dy= \int\limits_0^{\infty}u(0,y)k^b_{t}(x,y)dy.
\]
Letting $\epsilon\to 0$ gives the result. 

For $b=0$ we proceed a little differently. If we write $u=u_0+f(0),$ then $u_0$ also satisfies 
$(\pa_t-L_0)u_0=0$, $u_0(0,t) = 0$, and the same regularity conditions as those satisfied by $u$.
Arguing as above, 
\[
u_0(x,t)=\int\limits_{0}^{\infty}k^{0,D}_t(x,y)(f(y)-f(0))\, dy, 
\]
and hence
\[
u(x,t)=\int\limits_{0}^{\infty}k^{0}_t(x,y)f(y)\, dy.
\]
\end{proof}

As a special case we can demonstrate that the kernels $\{k^b_t:\: t>0\}$ have the semi-group property.
\begin{corollary} If $t,s$ are positive numbers and $b> 0,$ then
\[
k^b_{t+s}(x,y)=\int\limits_{0}^{\infty}k^b_t(x,z)k^b_s(z,y)dz.
\]
\end{corollary}
\begin{proof} For fixed $y,$ the kernel $k^b_t(\cdot,y)$ satisfies $\pa_tk^b_t(x,y)-L_bk^b_t(x,y)=0,$ and 
is $\CI$ in $[0,\infty)\times (0,\infty).$ This solution decays exponentially as $x\to\infty,$ so the maximum 
principle implies $s\mapsto k^b_{t+s}(x,y)$ is the unique smooth solution of this PDE with respect to the 
variables $(s,x),$ with initial data $k^b_t(\cdot,y).$ The proof now follows from Proposition \eqref{prop3}. 
\end{proof}
\begin{remark} A similar argument using the uniqueness of the solution to the Dirichlet problem when $b=0$ 
leads to a proof that 
\[
k^{0,D}_{t+s}(x,y)=\int\limits_{0}^{\infty}k^{0,D}_t(x,z)k^{0,D}_s(z,y)dz.
\]
It is also true that if $f\in\cC^{\infty}_c([0,\infty)),$ then
\[
\int\limits_{0}^{\infty}k^{0}_{t+s}(x,y)f(y)dy=
\int\limits_{0}^{\infty}k^{0}_t(x,z)\left[\int\limits_{0}^{\infty}k^{0}_s(z,y)f(y)dy\right]dz.
\]
\end{remark}

We conclude this section with:
\begin{proposition}\label{prop88}
If $f\in\cC^0([0,\infty))$ has an absolutely integrable Fourier transform, then 
\[
\lim_{b \to 0}\int_0^\infty k_t^b(x,y)f(y)\, dy  = \int_0^\infty k_t^0(x,y)f(y)\, dy.
\]
\end{proposition}

This convergence  is obvious  using the Fourier  representations (\ref{eq:hk0})
and (\ref{eq:hkfb})  of the kernels $k_t^b$  and $k_t^0.$ If $\hf$ is
absolutely integrable, then, for any fixed $0<T$ the convergence is uniform on
$[0,\infty)\times [0,T].$

\section{Mapping properties for the model operators}\label{s.mapping}
We can now prove the key fact that the diffusions associated to the
operators $L_b$ are semi-groups, which, for every $m \in {\mathbb N}\cup\{0\},$
preserve the spaces $\dot{\calC}^m([0,\infty)),$ see Remark~\ref{cmdot}.  We
prove this in two steps. The first is to show that the kernels map polynomials
to polynomials; then, decomposing an arbitrary $\calC^m$ initial condition into
a polynomial (its Taylor series) and a remainder term that vanishes to order
$m$, the fact that the solution with initial condition given by this remainder
term is also $\calC^m$ follows from the maximum principle.

Suppose that the initial data is a polynomial, $ f(x)=\sum_{j=0}^na_jx^j$. It
is clear that the formal exponential
\[
u = e^{ t L_b} f(x) = \sum_{\ell = 0}^\infty \frac{t^\ell}{\ell !} L_b^\ell f(x)
\]
makes sense and in fact is a polynomial which is the unique moderate growth
solution of (\ref{ivpb1}). We simply observe that $L_b^\ell x^j$ is a constant
multiple of $x^{j-\ell}$, hence this vanishes as soon as $\ell > j$. Thus the
sum on the right above is finite and is a polynomial in $(x,t)$.  It solves the
equation by the usual elementary calculation. Therefore, by the uniqueness
theorem, this is the only (exponentially bounded) solution of this problem, and
hence, by the representation formula, we must also have that
\[
u(x,t)=\int\limits_{0}^{\infty}k^b_t(x,y)f(y)\, dy.
\]

We now turn to the case of general $\cC^m$ data.
\begin{proposition}\label{prop99} Suppose that $f \in \calC^m([0,\infty))$ has compact support. Then
\[
u(x,t) = \int_0^\infty k_t^b(x,y) f(y)\, dy \in \calC^0([0,\infty)_t;\cC^m([0,\infty)_x)),
\]
solves the initial value problem~\eqref{eq:mhe}. The norm of the difference,
$\|u(\cdot,t)-f\|_{\cC^m}$ tends to zero as $t\to 0,$ and therefore $k_t^b$
extends to define a $\cdC^m$-semi-group for each $m = 0, 1, 2, \ldots$.
\label{pr:cmsemi}
\end{proposition}
\begin{remark}\label{cmdot} We let $\cdC^m([0,\infty))$ denote the closed subspace of
  $\cC^m([0,\infty))$ consisting of functions $f$ with 
  \[
    \lim_{x\to\infty} \pa_x^jf(x)=0\text{ for }0\leq j\leq m.
  \]
This is the closure of $\cC^m_c([0,\infty))$ in the $\cC^m$-norm.
\end{remark}

To prove the proposition, the following two lemmas are useful
\begin{lemma}\label{lem1} For $b\geq 0$ and $0 < \ell <L$, there is a constant $C_b$ so that if
$f\in\cC_c^0([0,L)),$  then, for $x>L,$ and $t<1,$ we have the estimate
\[
\left|\int\limits_{0}^{\infty} k^b_t(x,y)f(y)dy\right|\leq  C_b\|f\|_{\infty}\sqrt{L}\left(\frac{L}{x}
\right)^{\frac b2-\frac 14}\frac{e^{-\frac{(\sqrt{x}-\sqrt{L})^2}{t}}}{\sqrt{t}}.
\]
If $f\in\cC^0(l,\infty),$ is of tempered growth, then for $x<l,$ and $t<1,$
\[
|u(x,t)|\leq C_{b,N}\|f\|_{(N)}\frac{e^{-\frac{(\sqrt{l}-\sqrt{x})^2}{t}}}{t^b}.
\]
Here $N\in\bbN$ is chosen so that
\[
\|f\|_{(N)} :=\sup_{0<x}(1+|x|)^{-N}|f(x)|<\infty.
\]
\end{lemma}
The proofs are elementary using the asymptotic expansion for $\psi_b$ and are
left to the reader. It is also useful to have the following
\begin{lemma}
If  $b\geq 0,$ and $f\in\cC^m_c((0,\infty)),$ then
\[
u(x,t)=\int\limits_{0}^{\infty}k^b_t(x,y)f(y)\, dy,
\]
solves~\eqref{eq:mhe} and 
\[
\lim_{t\to 0^+}\|u(\cdot,t)-f\|_{\cC^m}=0.
\]
\end{lemma}
\begin{proof} It is clear that $\pa_tu-L_bu=0$ in $[0,\infty)\times(0,\infty).$
  Suppose that $\supp f\subset [l,L].$ The previous lemma shows that for any
  $\eta>0,$ the functions $u(\cdot,t)(1-\chi_{[l-\eta,L+\eta]}(\cdot))$
  converge uniformly to zero as $t\to 0^+.$ If $\eta<l/2,$ then $xy>l^2/2,$ for
  $(x,y)\in [l-\eta,L+\eta]\times\supp f,$ and we can therefore use the
  asymptotic expansion for $\psi_b$ to conclude that $u(\cdot,
  t)\chi_{[l-\eta,L+\eta]}(\cdot)$ converges uniformly to $f.$ As
  $f\in\cC^m_c((0,\infty)),$ we know from Lemma~\ref{difnlem} that
  \[
    \pa_x^ju(x,t)=\int\limits_0^{\infty}k^{b+j}_t(x,y)\pa_y^jf(y)dy,\text{ for
    }j\leq m.
  \]
The argument above applies to show the uniform convergence of the derivatives
$\pa_x^ju(\cdot,t)$ to $\pa_x^j f,$ for $1\leq j\leq m.$
\end{proof}

\begin{proof}[Proof of Proposition~\ref{prop99}]
  Fix $m$ and $f \in \calC^m_c[0,\infty).$ From the expressions~\eqref{psiexp1}
  and~\eqref{k0decmp} for $k_t^b$ it is clear that, if $f\in\cC^0,$ then
  $u\in\CI([0,\infty)\times (0,\infty)),$ and $\pa_tu-L_bu=0$ in
  $[0,\infty)\times (0,\infty).$ Choose a smooth cutoff function $\chi(x)$
  which equals $1$ near $x=0$, $|\chi(x)| \leq 1$, and which vanishes for $x
  \geq 1.$ Let
\[
q(x) = \sum_{j=0}^m \frac{f^{(j)}(0)}{j!}x^j
\]
be the Taylor polynomial of order $m$ for $f$; thus 
\[
\tilde{f}(x) = f(x) - \chi(x)q(x) = o(x^m), \qquad \mbox{as}\ x \searrow 0.
\]
We observe that
\[
\begin{split}
u_{\chi q}(x,t)&=\int_0^\infty k_t^b(x,y) \chi(y)q(y)\, dy \\
&= 
\int_0^\infty k_t^b(x,y) q(y)\,dy - \int_0^\infty k_t^b(x,y)(1-\chi(y))q(y)\, dy.
\end{split}
\]
By the remarks above, the first term on the right is a polynomial, and therefore
in $\calC^{\infty}([0,\infty)\times [0,\infty)).$ On the other hand,
$(1-\chi(y))q(y)$ is supported away from $y=0$, and it is elementary from the
regularity properties of $k_t^b$ proved in the last section and
Lemma~\ref{lem1} that
\[
\int_0^\infty k_t^b(x,y) (1-\chi(y))q(y)\, dy\in \calC^{\infty}([0,\infty)_x \times [0,\infty)_t),
\]
and tends locally uniformly to $(1-\chi(x))q(x)$ as $t\to 0^+.$ Together with
Lemma~\ref{lem1}, this shows that $u_{\chi q}(x,t)$ tends to $\chi q$ in
$\cC^l([0,\infty))$ for all $l\in\bbN.$ Thus it remains to prove that 
\[
\tilde{u}(x,t) = \int_0^\infty k_t^b(x,y)\tilde{f}(y)\, dy
\]
is also in $\calC^m$ and tends to $\tilde{f}$ in the $\cC^m$-norm as $t\to
0^+.$ 

First consider the case $m=0$. For every $\e > 0$ we can choose $\delta >
0$ so that $\tilde{f}(x) < \e$ when $x < \delta$. Decompose $\tilde{f}(x) =
\chi(x/\delta)\tilde{f}(x) + (1-\chi(x/\delta)) \tilde{f}(x)$. Applying $k_t^b$
to the second term gives a smooth function $v(x,t)$ in
$\CI([0,\infty)\times(0,\infty)),$ which satisfies $\pa_t v-L_bv=0.$
Lemma~\ref{lem2} shows that $v\in\cC^0([0,\infty)\times[0,\infty)),$ and tends
uniformly to $(1-\chi(x/\delta)) \tilde{f}$ as $t\to 0^+.$ As $k^b_t$ is
pointwise positive, and has integral $1$ for all $t>0,$ we see that
\[
\left|\int_0^\infty k_t^b(x,y)\chi(y/\delta)\tilde{f}(y)\, dy \right| \leq \e.
\]
Since $\e$ is arbitrary, this shows that 
\[
  \limsup_{t\to 0^+}\|\tu(\cdot,t)-f\|_{\cC^0}=1.
\]

If $m > 0$, then $\pa_x^j\tf(0)=0,$ for $1\leq j\leq m,$ and we can apply
Lemma~\ref{difnlem} to conclude that
\[
  \pa_x^j\tu(x,t)=\int\limits_0^{\infty}k^{b+j}_{t}(x,y)\pa_y^j\tf(y)dy,\text{
    for }1\leq j\leq m.
\]
The $\cC^0$-argument then applies to complete the proof.
\end{proof}
As a corollary of these results we can extend Lemma~\ref{difnlem} to general
data in $\cC^m_c([0,\infty)).$
\begin{corollary}\label{cor33} If $b\geq 0,$ and $f\in\cC^m_c([0,\infty)),$ and
  \[
    u(x,t)=\int\limits_0^{\infty}k^b_t(x,y)f(y)dy
  \]
 then
  \begin{equation}\label{diffjker}
    \pa_x^ju(x,t)=\int\limits_0^{\infty}k^{b+j}_t(x,y)\pa_y^jf(y),
\text{ for }1\leq j\leq m.
  \end{equation}
\end{corollary}
\begin{proof} The Proposition implies that $u$ is a solution to
  $\pa_tu-L_bu=0,$ which belongs to $\cC^m([0,\infty)\times[0,\infty)),$ and
  satisfies
  \[
    \lim_{t\to 0^+}u(\cdot,t)=f,
  \]
with convergence in the $\cC^m$-norm. One can therefore differentiate the
equation satisfied by $u$ to conclude that
\[
  \pa_t\pa_x^ju-L_{b+j}u=0\text{ for }1\leq j\leq m,\text{ and }
\lim_{t\to 0^+}\pa_x^ju=\pa_x^jf.
\]
For each $j,$ the right hand side
of~\eqref{diffjker} is another solution, $u_j$ to this initial value problem, that
also satisfies the hypotheses of the maximum principle. Thus $u_j=\pa_x^ju$ for
$1\leq j\leq m.$
\end{proof}
We can use this result to study the regularity properties of the operator
\[
g \mapsto K_t^bg(x)=\int\limits_0^t\int\limits_0^{\infty}k^b_{t-s}(x,y)g(y,s)dyds.
\]
A case of particular importance in applications is when $g \in \CI([0,\infty)\times(0,\infty))$,  
for example $g$ is any solution to~\eqref{eq:mhe}. We begin with a lemma.
\begin{lemma} Suppose that $b\geq 0,$ and $f\in\cC_b^{m}([0,\infty))$; then 
\[
u(x,t)=\int\limits_0^{\infty}k^b_t(x,y)f(y)\, dy
\]
satisfies
\[
\pa_t^ju(x,t)=\int\limits_0^{\infty}k^b_t(x,y)L_b^jf(y)\, dy
\] 
for $2j\leq m$ and $t>0.$
\end{lemma}
\begin{proof} Suppose that $m=2$ and $f$ vanishes to order $2$ at $x=0.$ Using
the fact that $(\pa_t-(L^{b}_y)^t)k_t(x,y)=0$ and a simple integration by parts
argument we easily see that
\[
\pa_tu(x,t)=\int\limits_0^{\infty}k^b_t(x,y)L_bf(y)dy.
\]
The second order vanishing of $f$ at $0$ implies that the boundary terms at $x=0$ vanish.

Choose a smooth function $\chi,$ which equals $1$ in $[0,1]$ and is supported in $[0,2].$ 
If $f$ does not vanish at $x=0$, then we let
\[
q_2(x)=\chi(x)\left(f(0)+f'(0)x+f''(0)\frac{x^2}{2}\right), \text{ and }\tf=f-q_2.
\]
From the argument above, 
\[
\pa_t\int\limits_0^{\infty}k^b_t(x,y)\tf(y)dy=\int\limits_0^{\infty}k^b_t(x,y)L_b\tf(y)dy,
\]
and then by Proposition~\ref{prop99},
\[
\pa_t\int\limits_0^{\infty}k^b_t(x,y)q_2(y)dy=\int\limits_0^{\infty}k^b_t(x,y)L_bq_2(y)dy.
\]
This argument can be applied inductively to obtain the case of general $m.$
\end{proof}
Combining this lemma with Corollary~\ref{cor33}, we conclude
\begin{corollary}\label{cor44}
Suppose that $b\geq 0,$ and $f\in\cC_b^{m}([0,\infty)),$ then 
\[
u(x,t)=\int\limits_0^{\infty}k^b_t(x,y)f(y)dy,
\]
satisfies
 \[
\pa_t^j\pa_x^ku(x,t)=\int\limits_0^{\infty}k^{b+k}_t(x,y)L_{b+k}^j\pa_y^kf(y)dy,
\] 
provided $2j+k\leq m$ and $t>0.$
\end{corollary}

We can now examine the regularity of $K^b_tg.$
\begin{proposition}\label{sptmders} If $b\geq 0,$ $T>0,$ and 
  \[
g\in \cC^0_b([0,\infty)\times [0,T))\cap\cC^m_b([0,\infty)\times (0,T)),  
  \]
then, for $0<t<T,$ the derivatives $\pa_t^j\pa_x^kK^b_tg$ are continuous for
$2j+k\leq m;$ moreover, if $m\geq 2,$ then 
\begin{equation}\label{inhompde}
(\pa_t-L_b)K^b_tg=g
\end{equation}
\end{proposition}
\begin{proof} For $t>2\epsilon>0,$ we define
\[
K^b_{t,\epsilon}g(x)=\int\limits_0^{t-\epsilon}\int\limits_0^{\infty}k^b_{t-s}(x,y)g(y,s)dyds.
\]
For $\epsilon>0$ it follows from the corollary, that the derivatives $\pa_t^j\pa_x^kK^b_{t,\epsilon}g$ 
exist, provided $2j+k\leq m$ and can be expressed as
\begin{multline*}
\pa_t^j\pa_x^k  K^b_{t,\epsilon}g(x)=
\int\limits_0^{\frac t2}\int\limits_0^{\infty}\pa_t^j\pa_x^kk^b_{t-s}(x,y)g(y,s)\, dyds+\\
\int\limits_{\frac t2}^{t-\epsilon}\int\limits_0^{\infty}k^{b+k}_{t-s}(x,y)L_{b+k}^j\pa_y^kg(y,s)\, dyds. 
\end{multline*}
If $2j+k\leq m,$ then $\pa_t^j\pa_x^k  K^b_{t,\epsilon}g(x)$ converges
locally uniformly to a continuous function. This establishes the existence of
these derivatives. If $m\geq 2,$ then differentiating, for any $0<\epsilon<t$ we have that
\begin{equation}\label{apprxpde}
  (\pa_t-L_b)K^b_{t,\epsilon}g=\int\limits_0^{\infty}k^b_{\epsilon}(x,y)g(y,t-\epsilon)dy.
\end{equation}
As $m\geq 2,$ the limiting function $K^b_tg$ has one time, and two spatial
derivatives, which are the limits of the corresponding derivatives of
$K^b_{t,\epsilon}g.$ Thus letting $\epsilon\to 0^+$ in~\eqref{apprxpde} gives \eqref{inhompde}. 
\end{proof}

We can finally extend the convergence result Proposition~\ref{prop88} to arbitrary data in $\cdC^0([0,\infty)).$ 
\begin{proposition}\label{prop88.1}
Fix $f\in\cdC^0([0,\infty)),$ then
\[
\lim_{b \to 0}\int_0^\infty k_t^b(x,y)f(y)\, dy  = \int_0^\infty k_t^0(x,y)f(y)\, dy.
\]
For any $T>0$ this convergence takes place in the $\cC^0([0,\infty)\times [0,T])$-topology.
\end{proposition}
\begin{proof} Choose a sequence $\{f_n\} \subset\cC^0_c([0,\infty))$ with all $\hat{f}_n \in L^1$, 
and such that $f_n$ converges uniformly to $f.$ Let $u^b$ denote the solutions to~\eqref{eq:mhe} 
with initial data $f$ and $u^b_n$ the solutions with initial values $f_n.$ The maximum principle 
implies that, for any $n,$ 
\[
\|u^b-u^0\|_{\cC^0([0,\infty)\times [0,T])}\leq 2\|f-f_n\|_{\infty}+
\|u^b_n-u^0_n\|_{\cC^0([0,\infty)\times [0,T])}.
\]
Given $\epsilon>0,$ fix some $n$ so that $\|f-f_n\|_{\infty}<\epsilon.$  Applying Proposition~\ref{prop88} 
to the $u^b_n$ gives a positive $b_0$  so that if $b<b_0,$ then 
\[
\|u^b_n-u^0_n\|_{\cC^0([0,\infty)\times [0,T])}<\epsilon,
\]
as $\epsilon>0$ is arbitrary, this completes the proof of the proposition.
\end{proof}

\section{Perturbation estimates}\label{s.pests}
In order to use the model heat kernels $k^b_t$ in perturbative constructions for the heat kernels of 
general Wright-Fisher operators, it is necessary to prove estimates in $\cC^\ell$ for every $\ell \geq 0$ 
for operators of the form
\begin{equation}
g \mapsto  A^b_tg(x)=\int\limits_{0}^t \int\limits_{0}^{\infty}k^b_{t-s}(x,z)h(z)z\pa_zg(z,s)\, dzds,
\end{equation}
were $h \in \CIc([0,\infty))$ is fixed.  

Corollary~\ref{cor33}, the Leibniz formula and the mapping results established in the previous section 
show that if $g \in \cC^\ell$, then 
\begin{multline}\label{leibfrm0}
\pa_x^\ell A^b_tg=A^{b+\ell}_t\pa_y^\ell g\, + \\
\sum_{j=1}^{\ell}\left(\begin{matrix}l\\j\end{matrix}\right)
\int\limits_{0}^{t}\int\limits_0^{\infty}k^{b+\ell}_{t-s}(x,z)\pa_z^{j}(zh(z))\pa_z^{\ell+1-j}g(z,s)\, dzds.
\end{multline}
Hence to estimate $A_t^b g$ on $\cC^\ell$ for arbitrary $\ell \geq 0$ it suffices to prove mapping properties 
for $A^{b'}_t$ on $\cC^0$ for all $b'\geq  0$.

The following lemma is the key to all that follows:
\begin{lemma}\label{lem2} There is a constant $C_b$, defined for each $b \geq 0$ and uniformly
bounded for $b$ in any compact interval $[0,B]$, such that if $f\in\cC_b^0([0,\infty))$ and 
\[
u(x,t)=\int\limits_0^{\infty}k^b_t(x,y)f(y)\, dy.
\]
then
\[
|\pa_z u(z,s)|\leq C_b\frac{\|f\|_{\infty}}{s+\sqrt{zs}}.
\]
\end{lemma}
\begin{proof}
It is enough to prove that for each $b > 0$ there is a constant $C_b$ so that 
\begin{equation}
\phi_{s,b}(z) : =\int\limits_{0}^{\infty}\left|\frac{\pa k^b_s}{\pa z}(z,y)\right|\, dy \leq\frac{C_b}{s+\sqrt{sz}} =
\frac{C_b}{s (1 + \sqrt{z/s})}, 
\label{bscest1}
\end{equation}
and that $C_b$ is uniformly bounded above on any interval $(0,B]$, i.e.\ it does not depend on a positive 
lower bound for $b$. The case $b=0$ is then obtained from a separate limiting argument.

We first compute that 
\[
\phi_{s,b}(z)=\frac{1}{s}\int\limits_{0}^{\infty}\left(\frac{y}{s}\right)^be^{-\frac{z+y}{s}}
\left|\left(\frac{y}{s}\right)\psi_b'\left(\frac{zy}{s^2}\right)-
\psi_b\left(\frac{zy}{s^2}\right)\right|\frac{dy}{y}.
\]
Set $w=y/s$ and $\lambda=z/s,$ so that $\phi_{s,b}(z)=\frac{1}{s}\varphi\left(\frac{z}{s}\right),$ where
\[
\varphi(\lambda) = \int\limits_{0}^{\infty}w^{b-1}e^{-w}e^{-\lambda}|w\psi_b'(\lambda
w)-\psi_b(\lambda w)|\, dw.
\]
Hence we only need prove that $\varphi(\lambda) \leq C_b/(1 + \sqrt{\lambda})$. 

Using the asymptotic formul\ae
\begin{eqnarray*}
\psi_b(w) & \sim & \frac{w^{\frac14-\frac{b}2}e^{2\sqrt{w}}}{\sqrt{4\pi}}(1+O(w^{-\frac12}))\qquad \text{and} \\
\psi_b'(w) & \sim & \frac{w^{-(\frac14+\frac{b}2)}e^{2\sqrt{w}}}{\sqrt{4\pi}}(1+O(w^{-\frac12})),
\end{eqnarray*}
we see that $\varphi(\lambda) \leq C_{b,\Lambda}$ for $\lambda \leq \Lambda$.  It is less obvious that this 
constant is bounded as $b\to 0$ since  the integrand appears to become nonintegrable at $w=0$ in that limit. 
To analyze this, define $\tpsi_b(z)=\psi_b(z)-\psi_b(0)$ and recall that $\psi_b(0) = 1/\Gamma(b)$, so that
\[
\varphi(\lambda) \leq  \int\limits_{0}^{\infty}w^{b-1}e^{-w}e^{-\lambda}\left[|w\psi_b'(\lambda
w)|+|\tpsi_b(\lambda w)|+\frac{1}{\Gamma(b)}\right]\, dw.
\]
The integral of the last term in brackets is identically equal to $e^{-\lambda}$. As for the other two terms, note 
that the coefficients in the error terms in these asymptotic formul\ae\ are bounded as $b \to 0$, so the 
integral from $1$ to $\infty$ converges uniformly, independently of $b$. If $w \leq 1$ (and $\lambda$ bounded),
these terms are ${\mathcal O}(w)$, hence this part of the integral is also uniformly bounded.  Hence 
$C_{b,\Lambda}$ is uniform in $b \in [0,B]$ for any fixed $B, \Lambda$. 

Now consider what happens when $\lambda\to\infty$. Suppose first that $b > 0$. Break the integral defining
$\varphi$ into the sum $J_b' + J_b''$, where $J_b'$ is the integral from $0$ to $1/\sqrt{\lambda}$ and 
$J_b''$ is the integral from $1/\sqrt{\lambda}$ to $\infty$.  It is straightforward that $J_b'\leq Ce^{-c\lambda}$,
for $c,C>0$ which are independent of $b$. For the other part use the asymptotics of $\psi_b$ and $\psi_b'$
to get
\[
J_b''\leq C_b\sqrt{\lambda}\int\limits_{0}^{\infty} \left(\frac{w}{\lambda}\right)^{\frac b2+\frac14}
e^{-\lambda(1-\sqrt{\frac{w}{\lambda}})^2} \left|1-\sqrt{\frac{w}{\lambda}}\right|\, \frac{dw}{w}.
\]
Changing variables to $y=\sqrt{w/\lambda}-1$ transforms this to  
\[
J_b''\leq C_b\sqrt{\lambda}\int\limits_{-1}^{\infty}
\left(1+y\right)^{b-\frac 12}e^{-\lambda y^2} |y|\, dy,
\]
which now, by Laplace's method, satisfies $J_b''\leq \frac{C_b}{\sqrt{\lambda}}$. 
This proves the estimate for $b > 0$ bounded away from $\infty$. 

To finish the proof, observe that for any $\eta>0$ and fixed $z,s > 0$, we have
\[
\lim_{b\to 0^+}\int\limits_{\eta}^{\infty}|\pa_zk^b_s(z,y)|\, dy= \int\limits_{\eta}^{\infty}|\pa_zk^{0,D}_s(z,y)|\, dy
\]
Since the constant $C_b$ is uniformly bounded as $b \to 0$, there is a constant $C_0$ so that for any $\eta>0$, 
\[
\int\limits_{\eta}^{\infty}|\pa_zk^{0,D}_s(z,y)|\, dy\leq \frac{C_0}{s(1+\sqrt{z/s})}.
\]
The right hand side is independent of $\eta$, so letting $\eta\to 0$ gives the same estimate for the integral on  
all of ${\mathbb R}^+$.  Setting $z=0$ shows that
\[
\int\limits_{\eta}^{\infty}|(\pa_zk^{0,D}_s)(0,y)|\, dy=\frac{\psi_2(0)}{s}.
\]
Since 
\[
\int\limits_0^{\infty}k^0_s(z,y)f(y)\, dy = \int\limits_0^{\infty}k^{0,D}_s(z,y)(f(y)-f(0))\, dy+f(0),
\]
this completes the case when $b=0$ as well.
\end{proof}

\subsection*{Mapping properties of $(A_t^b)^j$}
Now we turn to estimates of $A^b_t$ and its iterates.  Because the proofs of
the next two Propositions are quite technical, we state the results here and
relegate their proofs to Appendix~\ref{pertthryprfs} at the end of the paper.

We assume that $g \in \cC^0([0,\infty)\times [0,T])\cap \cC^1((0,\infty)\times
(0,T])$ with a very specific blowup for $\pa_xg$ as $x\to 0^+$, as suggested by
Lemma~\ref{lem2}.
\begin{proposition}\label{prop1515} Define the sequence of constants 
\[
d_j=\frac{(\pi)^{\frac{j+1}{2}}}{\Gamma\left({\frac{j+1}{2}}\right)}.
\]
For any smooth $h$ with $\supp h\subset [0,L]$, if $g\in\cC^0([0,\infty)\times [0,T])\cap\cC^1((0,\infty)\times (0,T])$
satisfies 
\[
|\pa_xg(x,t)|\leq\frac{M}{\sqrt{xt}},
\]
then
\begin{equation}\label{Aestj}
|(A^b_t)^jg(x)|\leq \frac{2d_{j-1}}{j}MC_b^{j-1}(\sqrt{L}\|h\|_{\infty})^jt^{\frac{j}2}
\end{equation}
and
\begin{equation}\label{dAestj}
|\pa_x(A^b_t)^jg(x)|\leq \frac{d_{j}M(C_b\sqrt{L}\|h\|_{\infty})^jt^{\frac {j-1}2}}{\sqrt{x}},
\end{equation}
where $C_b$ is the constant appearing in Lemma~\ref{lem2}.
\end{proposition}
\begin{remark}\label{b0Atrmrk} The hypotheses of this proposition imply that
\[
|yh(y)\pa_yg(y,s)|\leq\frac{\sqrt{y}\|h\|_{\infty}M}{\sqrt{s}},
\]
so we could replace $k^0_{t-s}$ by $k^{0,D}_{t-s}$ in the definition of $A^0_t$. With this
choice, $A^0_tg(0)=0$. However, $k^{0,D}_t$ does not satisfy Corollary~\ref{cor44},
and its use would also complicate the derivation of the estimate for $\pa_x^jA^0_tg(x)$,
so it is simpler to use the kernel $k^0_{t-s}$ in the definition of $A^0_t$.
\end{remark}

For the higher norm estimates for $A^b_t$ and its iterates, it is useful to introduce,
for $T>0$ and $\ell \in\bbN,$ the norms
\[
\|g\|_{\cC^{\ell,\infty}[0,T]}=\max_{0\leq t\leq T}\|g(\cdot,t)\|_{\cC^\ell([0,\infty))}
\]
The maximum principle and \eqref{leibfrm0} immediately give the
\begin{lemma}\label{higdest} 
If $g\in\cC^{\ell}_b([0,\infty)\times[0,T]),$ and $1\leq p\leq \ell-1,$ $j\in\bbN,$ then
\[
|\pa_x^p(A^b_t)^jg(x)|\leq \frac{t^j}{j!}[2^p\|xh\|_{\cC^{p}}]^j \|g\|_{\cC^{p+1,\infty}[0,T]}.
\]
\end{lemma}

\begin{proposition}\label{atbests2} 
Define the constants $D_j$ inductively by 
\[
D_{j}=2D_{j-1}\frac{\Gamma(\frac{j}{2})}{\Gamma(\frac{j+1}{2})},
\]
and $D_0 = 1$, so that $D_j\leq C\frac{2^j}{\Gamma(\frac{j+1}{2})}$ for all $j \geq 0$. 
Let $h$ be any smooth function with $\supp h\subset [0,L]$ and suppose that the function
\[
g\in\cC^{\ell}([0,\infty)\times [0,T])\cap\cC^{\ell+1}((0,\infty)\times [0,T]),
\]
satisfies 
\begin{equation}\label{dgestl}
|\pa^{\ell +1}_xg(x,t)|\leq\frac{M}{\sqrt{xt}}
\end{equation}
for some constant $M > 0$ (depending on $g$).  Then there are constants $C_{T,L,\ell,b}, C'_{T,L,\ell,b}$ so that, 
for $j\in\bbN,$ we have
\begin{equation}\label{Aestjl}
|\pa^{\ell}_x[A^b_t]^jg(x)|\leq \frac{2D_{j-1}}{j}(M+\|g\|_{\cC^{\ell,\infty}[0,T]}) (C'_{T,L,\ell,b}\|h\|_{\cC^\ell})^jt^{\frac {j}2},
\end{equation}
and
\begin{equation}\label{dAestjl}
|\pa^{\ell+1}_x[A^b_t]^jg(x)|\leq \frac{D_{j}(M+\|g\|_{\cC^{\ell,\infty}[0,T]})(C_{T,L,\ell,b}\|h\|_{\cC^{\ell+1}})^jt^{\frac {j-1}2}}{\sqrt{x}}.
\end{equation}
\end{proposition}

\subsection*{Off-diagonal estimates for $(A_t^b)^j$}
The final set of estimates we derive for the operators $A_t^b$ involve the
off-diagonal behavior of the Schwartz kernel of the infinite sum
\begin{multline*}
\sum_{j=0}^\infty (A_t^b)^j k_t^b (x,y) = \sum_{j=0}^\infty \int_0^t \cdots \int_0^{s_1} \int_0^\infty \cdots \int_0^\infty
A_{t-s_j}^b(x,z_1) \times \\  A_{s_j-s_{j-1}}^b(z_1,z_2) \ldots  A_{s_2 - s_1}^b(z_{j-1},z_j) k_{s_1}^b(z_j,y)\, ds_1 \ldots ds_j dz_1 \ldots dz_j.
\end{multline*}
The most precise bounds are best described in a blown-up space. For our
purposes, the slightly cruder estimates derived here suffice.

We begin by observing that if $|x-y| \geq \alpha > 0$, then $k_t^b(x,y) \leq C
y^{b-1}e^{-c/t}$ (where $c$ is any number less than $(\sqrt{x} -
\sqrt{y})^2$). In fact, $k_t^b(x,y) = y^{b-1}F(t,x,y)$ where $F$ is smooth away
from the diagonal $x=y$ but up to $x=0$ and $y=0$, and satisfies $|F| \leq C
e^{-c/t}$ where $c$ depends only on $|x-y|$. This is proved using the explicit
expression for $k_t^b$, and separately considering the behavior in the regions
$xy > t^2$ and $xy < t^2$.  Using this we can now prove the
\begin{proposition}
  Fix $\alpha \geq 0$, and $k, \ell \in {\mathbb N}$; then there exist
  constants $C, c, B > 0$ depending on $\alpha$, $||h||_\infty$, $L$, $k$ and
  $\ell$ such that if $|x-y| \geq \alpha$ and $0 \leq x, y \leq L$, then
\[
|\pa_x^k (y\pa_y)^\ell [(A_t^b)^j k_t^b](x,y)| \leq C D_j B^j e^{-c/t} y^{b-1},
\]
where $D_j$ are the constants appearing in Proposition \ref{atbests2}. (These
may be replaced by the constants $d_j$ from Proposition \ref{prop1515} when $k
= \ell = 0$.)
\end{proposition}
\begin{proof}
  As just indicated, this assertion is clear from the formula for $k_t^b$ when
  $j=0$, so we proceed by induction, assuming that we have proved it for all
  powers up to some $j$.

  First use a smooth partition of unity to decompose $k_t^b = k_t^{1,b} +
  k_t^{2,b}$ where $k_t^{1,b}$ is supported in the region $|x-y| \leq \alpha/2$
  and $y^{1-b} k_t^{2,b}$ is $\calC^\infty$ when $x,y,t \geq 0$ and satisfies
  $|y^{1-b}k_t^{2,b}(x,y)| \leq C e^{-c/t}$ for all $x,y$. Using Proposition
  \ref{prop1515} (or Proposition \ref{atbests2} for the higher derivatives), we
  have that
\[
|(A_t^b)^{j+1} k_t^{2,b}(x,y)| \leq C d_j B^{j+1} e^{-c/t}y^{b-1}.
\]

To estimate the other term, let us begin by writing
\[
\begin{split}
A_t^b \circ k_t^{1,b}(x,y) & = \int_0^t \int_0^\infty k_{t-s}^b(x,z) h(z) z\pa_z k_s^{1,b}(z,y)\, dz ds \\
& = \int_0^{t/2} \int_0^\infty -(z\pa_z + 1) (k_{t-s}^b(x,z)h(z)) k_s^{1,b}(z,y)\, dz ds  \\ 
& + \int_{t/2}^t \int_0^\infty k_{t-s}^b(x,z)h(z) z\pa_z k_s^{1,b}(z,y)\, dz ds.
\end{split}
\]
The integration by parts used to obtain the first term on the right is valid
because $z k_{t-s}^b(x,z)k_s^{1,b}(z,y) \leq Cz^b$ and $h(z)$ has compact
support. For this first term, use that $|x-z| \geq \alpha/2$ and $|z-y| \leq
\alpha/2$ to get
\[
|(z \pa_z + 1) k_s^b(x,z)| \leq C e^{-c/t}z^{b-1}, \quad \mbox{and} \quad |k_s^{1,b}(z,y)| \leq C s^{-1/2} y^{b-1},
\]
which shows that this term is bounded by $B_1 e^{-c/t} y^{b-1}$.  
The second term is bounded similarly, using 
\[
|\pa_z k^{1,b}_s(z,y)| \leq B_2 s^{-3/2}y^{b-1}.
\]  
Taken together, we obtain that
\[
|A^b_t k_t^b(x,y)| \leq B e^{-c/t} y^{b-1}
\]
where $B$ depends only on the quantities indicated.  Finally, use Propositions \ref{prop1515} and \ref{atbests2} to obtain
\[
|(A^b_t)^j \circ (A^b_t k^{1,b}_t)| \leq C d_j B^{j+1} e^{-c/t} y^{b-1},
\]
as claimed. The estimates for the higher derivatives are proved similarly.
\end{proof}

From these estimates we now obtain the
\begin{corollary}
\label{kernest1}
The kernel $\tq^{\,b}_t(x,y)$ of the operator
\[
f \mapsto \tQ_t^bf=\sum_{j=0}^\infty(A_t^b)^j\int\limits_0^{\infty}k_t^{b}(\cdot,y)f(y)dy
\]
can be written as
\[
\tq^{\, b}_t=y^{b-1}\tq^{\,b,\reg}_t(x,y),\text{ where }  \tq^{\,b,\reg}_t\in\CI([0,1]\times[0,1]\times (0,\infty)).
\]
For each $d>0,$ there is a constant $c_d>0,$ so that in the off-diagonal region $\{(x,y):\: |x-y|>d\},$ we have the estimate
\begin{equation}\label{offdiag2}
|\tq^{\, b}_t(x,y)|\leq C e^{-\frac{c_d}{t}} y^{b-1},
\end{equation}
with analogous estimates for the derivatives $\pa_x^k(y\pa_y)^l\tq^{\, b}_t(x,y).$
\end{corollary}

\section{Construction of the heat kernel for a general Wright-Fisher operator}
\label{htker4L}
After this long excursion into the analysis of the model problems we are now prepared to construct the heat kernel for 
the full generalized Wright-Fisher operator 
\[
L=y(1-y)\pa_y^2+b(y)\pa_y.
\]
Let
\[
y=\sin^2\sqrt{x_\ell}\text{ and }1-y=\sin^2\sqrt{x_r}.
\]
According to the discussion in Section~\ref{modprob1}, pulling back $L$ to these coordinate charts gives 
\[
L_\ell=x_l\pa_{x_\ell}^2+b_0\pa_{x_\ell}+x_\ell c_l(x_\ell)\pa_{x_\ell}\text{ and }  L_r=x_r\pa_{x_r}^2+b_1\pa_{x_r}+x_rc_r(x_r)\pa_{x_r},
\]
where
\[
b_0=b(0)\text{ and }b_1=-b(1).
\]

Suppose that $u$ solves~\eqref{genWFivp}, with $u(y,0)=f(y).$ Then on the interval $[0,\left(\frac{\pi}{2}\right)^2)$ the functions 
\[
u_\ell (x_\ell,t)=u(\sin^2\sqrt{x_\ell},t)\text{ and } u_r(x_r,t)=u(1-\sin^2\sqrt{x_r},t),
\]
satisfy
\[
\begin{split}
& (\pa_t-L_\ell)u_\ell=0\text{ with }u_\ell(x_\ell,0)=f(\sin^2\sqrt{x_\ell})\text{ and }\\
&(\pa_t-L_r)u_r=0\text{ with }u_r(x_r,0)=f(1-\sin^2\sqrt{x_r}).
\end{split}
\]
It is clear that the symmetry $y\to 1-y$ carries the left end to the right and vice-versa, so to simplify this discussion, we 
focus on the left end. We use $x$ to denote $x_\ell,$ $b$ to denote $b_0,$ and we choose a smooth cutoff function $\varphi$ so that
\[
\varphi(x)=
\begin{cases} 1\text{ for }x\in [0,\left(\frac{\pi}{2}\right)^2-2\eta]
\\ 0 \text{ for }x>\left(\frac{\pi}{2}\right)^2-\eta,
\end{cases}
\]
where $\eta>0$ is small. With $h(x)=c(x)\varphi(x),$ we now focus attention on
\[
\tL=L_{b_0}+xh(x)\pa_x.
\]
A solution to~\eqref{genWFivp}, pulled by via $x_\ell$ satisfies:
\[
(\pa_t-\tL)\tu=0\text{ on }[0,\left(\frac{\pi}{2}\right)^2-2\eta].
\]

To build a parametrix for the heat kernel that has the correct boundary
behavior near to $y=0,$ consider the initial value problem:
\begin{equation}\label{mdlprb2}
(\pa_t-\tL)\tu=0\text{ with }\tu(x,0)=\tf(x)\in\cC^0_c([0,\infty).
\end{equation}
For our application $\tf$ is obtained from $f$ by pullback and multiplication by
a smooth cut-off. Multiply the equation by $k^{b}_t(x,y)$  and  integrate to obtain:
\[
\tu(x,t)-\int\limits_0^t\int\limits_0^{\infty}k^{b}_{t-s}(x,y)yh(y)\pa_y\tu(y,s)dyds=
\int\limits_0^{\infty}k^{b}_t(x,y)\tf(y)\, dy,
\]
or equivalently
\[
\tu(x,t)-A^b_t\tu(y,s)= \int\limits_0^{\infty}k^{b}_t(x,y)\tf(y)\, dy.
\]
If $b=0,$ then as noted in Remark~\ref{b0Atrmrk}, we could replace $k^0_{t-s}$ in the definition of 
$A^0_t$ with $k^{0,D}_{t-s},$ which shows that, as expected, $\tu(0,t)=\tf(0)$ for all $t\geq 0.$

It is straightforward to solve \eqref{mdlprb2} using the estimates from Section~\ref{s.pests}. Indeed, 
the solution can be expressed as a convergent Neumann series:
\[
\tu(x,t)=(\Id-A^b_t)^{-1}\tg(x,t)=\sum_{j=0}^{\infty}(A^b_t)^j\tg, 
\]
where
\[
\tg(x,t)=\int\limits_0^{\infty}k^b_t(x,y)\tf(y)\, dy.
\]
Let us denote the operator $\tf\to(\Id-A^b_t)^{-1}\tg$ by $\tQ^b_t\tf.$  Later on we shall need to distinguish 
the operator constructed near $x=0$ from the one near $x=1$, and at that point we shall write them as 
$\tQ^{b_0}_{t,\ell}$ and $\tQ^{b_1}_{t,r}$, and denote their kernels by $\tq^{\, b_0}_{t,\ell}, \tq^{\,b_1}_{t,r}$, respectively. 

An immediate consequence of Propositions~\ref{prop1515}, and~\ref{atbests2} is
that if $f\in\cC^\ell_c([0,\infty)),$ then this sum converges uniformly in the
topology of $\cC^0([0,T];\cC^\ell([0,\infty)),$ for any $T>0$, and
Propositions~\ref{prop99},~\ref{prop1515}, and~\ref{atbests2} show that for
such data,
\[
\lim_{t\to 0^+} \pa_x^j\tu(\cdot,t)=\pa_x^j\tf\text{ for }j\leq l,
\]
where the convergence is with respect to the $\cC^0$ topology.

The regularity of these solutions shows that if $\tf\in\cdC^l([0,\infty)),$ then
\begin{equation}\label{gbreg}
\tg\in\cC^0([0,\infty);\cC^\ell([0,\infty))\cap \cC^{\infty}([0,\infty)\times (0,\infty)).
\end{equation}
As a consequence of Proposition~\ref{sptmders}, we see that for $0<\epsilon<t,$
the sum defining $\tQ_t\tf$ actually converges uniformly in the $\CI$-topology. Hence
\begin{equation}\label{gbreg1}
\tu\in\cC^0([0,\infty);\cC^\ell([0,\infty))\cap \cC^{\infty}([0,\infty)\times (0,\infty)),
\end{equation}
as well, and differentiating shows that
\[
(\pa_t-L_b)\tu=xh(x)\pa_x\tu \Longleftrightarrow   (\pa_t-\tL)\tu=0.
\]

We have now solved the problem `exactly' near each endpoint, and the next step
is to paste together these left and right solution operators to obtain an
infinite order parametrix for $\pa_t - L,$ which has remainder term vanishing
identically near the $x=0$ and $x=1$. To accomplish this, define maps
$\phi_\ell$ and $\phi_r$ near $x=0$ and $1$ which put the leading part of
$L_{\WF}$ into the model form $x\pa_x^2$. Pulling back the operator in
\eqref{WFnrmfrm2} gives exactly the operator $L$ in the intervals $[0,\frac
89],$ and $[\frac 19,0]$.  Now choose cutoffs
$\varphi_l,\varphi_r,\varphi_0\in\CI[0,1]$ so that
\[
\begin{split}
    &\varphi_l(x)=1\text{ for }x\in [0,\frac{11}{16}],\quad\supp\varphi_l\subset
    [0,\frac 34],\\
&\varphi_r(x)=1\text{ for }x\in [\frac{5}{16},1],\quad\supp\varphi_r\subset
    [\frac 14,1],\\
&\varphi_0(x)=1\text{ for }x\in [0,\frac 38],\quad\supp\varphi_0\subset
    [0,\frac 58].
\end{split}
\]
and define the parametrix 
\begin{multline*}
q_t(x,y)= \varphi_0(x)\tq^{\, b_0}_{t,l}(\phi_l(x),\phi_l(y))\varphi_l(y)|\phi'_l(y)| + \\
(1-\varphi_0(x))\tq^{\, b_1}_{t,r}(\phi_r(x),\phi_r(y))\varphi_r(y)|\phi'_r(y)|.
\end{multline*}
The heat kernel for $L$ is determined symbolically to infinite order as $t
\searrow 0$ near the diagonal away from $x,y=0$, or $1$. Furthermore, the kernel and
all its derivatives tend to zero like $e^{-c/t}$ away from the diagonal and
away from $y=0$ or $1.$ Hence in the overlap region,
$\supp\varphi_0\cap\supp(1-\varphi_0),$ the two terms agree to all orders.

Now set 
\[
e_t(x,y)=[\pa_t - (x(1-x)\pa_x^2+b(x)\pa_x)]q_t(x,y);
\]
by construction, this satisfies
\[
\supp e_t(x,y)\subset [\frac 38,\frac 58]\times [0,1]\times [0,\infty).
\]
Since this error term is obtained by applying derivatives in the first ($x$)
variable to $q_t(x,y)$, Corollary~\ref{kernest1} shows that
$y^{1-b_0}(1-y)^{1-b_1}e_t \in \calC^\infty([0,1]\times [0,1]\times(0,\infty))$
and vanishes, along with all its derivatives, like $e^{-c/t}$ for some $c>0$,
on $[0,1]\times [0,1] \times [0,\infty)$ disjoint from the diagonal at
$t=0$. In a neighborhood including the diagonal it vanishes, in the
$\CI$-topology, faster than $O(t^N)$ for any $N\in\bbN.$

Given a function $f\in\cC^m([0,1]),$  set
\[
u_0(x,t)=\int\limits_{0}^1q_t(x,y)f(y)\,dy.
\]
Propositions~\ref{prop1515} and~\ref{atbests2} show that
$u_0(x,t)\in\cC^0([0,\infty);\cC^m([0,1])),$ and
\[
[\pa_t- L]u_0(x,t)=v(x,t),\quad\lim_{t\to 0^+}u_0(x,t)=f(x),
\]
where $v(x,t)$ is smooth, vanishes in $\{[0,\frac 38]\cup[\frac 58,1]\}\times
[0,\infty)$ and tends rapidly to zero as $t\to 0^+.$ This is even true for
$\pa_x^j v(x,t)$ for all $j\in\bbN.$

We now complete the construction of the solution operator for the generalized
Wright-Fisher operator. Define, for $\epsilon>0$, 
\begin{equation}
\begin{split}
  Q_t^{\epsilon}g &=\int\limits_{0}^{t-\epsilon}\int\limits_0^1q_{t-s}(x,y)g(y,s)\, dyds\\
E_tg&=-\int\limits_{0}^{t}\int\limits_0^1e_{t-s}(x,y)g(y,s)\, dyds,
\end{split}
\end{equation}
and write $Q_t^0$ simply as $Q_t$. If 
\begin{equation}\label{greghyp2}
g\in\cC^0([0,\infty)\times [0,\infty))\cap\cC^2([0,\infty)\times (0,\infty)),
\end{equation}
then the estimates above imply that $Q_t^{\epsilon}g \to Q_tg$ in the $\cC^2$
topology for any $t > 0$, and hence
\[
(\pa_t-L)Q_tg=\lim_{\epsilon\to 0^+}(\pa_t-L)Q_t^{\epsilon}g=(\Id-E_t)g.
\]

The inverse of $(\Id-E_t)$ is an operator of the same type, and we write it as
$(\Id-H_t)$; here $H_t$ is represented by a kernel $h_t(x,y).$ Note that
\[
(\Id-E_t)(\Id-H_t)= (\Id - H_t)(\Id - E_t) = \Id,
\]
so
\[
H_t = -E_t + E_t H_t = -E_t + H_t E_t \Longrightarrow H_t = -E_t - E_t^2 + E_t H_t E_t.
\]
The last identity shows that $h_t$ same regularity properties as $e_t$.  In
particular, it vanishes, along with all derivatives, to infinite order as $t\to
0,$ and behaves like $y^{b_0-1}$ along $y=0,$ and $(1-y)^{b_1-1}$ along $y=1.$

Setting $g=(\Id-H_t)v,$ then $u_1=Q_tg$ satisfies
\[
(\pa_t-L)u_1=v\quad \text{ and }\qquad \lim_{t\to 0}u_1(x,t)=0.
\]
The true solution $u$ is the difference
\[
u=u_0-u_1=Q_t(\Id-E_t+H_tE_t)(f\otimes\delta(t)),
\]
or equivalently,
\begin{equation}\label{exctsoln}
u=\int\limits_0^1q_t(x,y)f(y)\, dy+\int\limits_0^t\int\limits_0^1q_{t-s}(x,z)
\int\limits_0^1h_s(z,y)f(y)\, dydzds.
\end{equation}
This was derived under the assumption that $g$ satisfies~\eqref{greghyp2}. If
$f\in\cC^0([0,1]),$ then the fact that the solution can be represented the same
way follows from the mapping properties of $Q_t.$

\begin{definition}
We let $\cQ_tf$ denote the operator on the right hand side
in~\eqref{exctsoln}, and $\hq_t(x,y)$ its kernel.
\end{definition}

Note that since $y^{1-b_0}(1-y)^{1-b_1}h_t$ decays rapidly decreasing as $t
\searrow 0$ in $\CI([0,1]\times[0,1])$, the kernel $q_t$ already defines a
complete parametrix for $(\pa_t-L)^{-1}$ away from $y=0$ and $y=1.$ The
contributions of $h_t$ along $y=0$ and $y=1$ are essential for $\cQ_t$ to
satisfy the adjoint boundary conditions~\eqref{eq:adjbc},
and hence to provide the solution operator for the adjoint problem. 

Putting together all the information we have obtained about the kernel for
$\cQ_t,$ and using the uniqueness for the regular solution of~\eqref{genWFivp},
we can now state the fundamental
\begin{theorem}\label{bscexstc} For each $m\in\bbN\cup \{0\}$ the operators $\cQ_t$ define 
  positivity preserving semi-groups on $\cC^m([0,1]).$ The function
  $u(x,t)=\cQ_tf(x)$ satisfies~\eqref{genWFivp}.  Moreover, for $f\in\cC^m$,
\[
\lim_{t\to 0^+}\|\cQ_tf-f\|_{\cC^m([0,1])}=0.
\]
\end{theorem}

One of the reasons for having worked hard to establish convergence of the Neumann series used in
the construction of the solution operator $\cQ_t$ is that we can estimate how good of an approximation
the partial sums are.

\begin{proposition}
Fix $N \geq 0$ and define the operator $Q_{t,N}$ by pasting together the finite sums
\[
\sum_{j=0}^N (A_t^b)^j k_t^b
\]
using the solution operators for the models at the left and right endpoints of
$[0,1]$, and denote its kernel by $q_{t,N}$.  For any $f \in \calC^m([0,1])$,
let $u(x,t)$ denote the exact solution to~\eqref{genWFivp}; for $N>m$ and
$2j\leq m,$ the function
\begin{equation}
u_N(x,t) = \int_0^1 q_{t,N}(x,y)f(y)\, dy
\end{equation}
satisfies
\[
||\pa_x^j[u_N(\cdot,t) - u(\cdot,t)]||_{\calC^0([0,1])} \leq C t^{\frac{N+1-2j}{2}}
\]
where the constant $C$ can be estimated in terms of $j, N$, and the
coefficients of $L$.  In particular, when $m= N = 0$ then
\[
\sup_{x \in [0,1]}  |u_0(x,t) - f(x)| \leq C\sqrt{t}.
\]
\end{proposition}

We recall the changes of variables
\begin{equation}
\begin{split}
\sqrt{x_l}=\sin^{-1}\sqrt{x} &\quad \sqrt{y_l}=\sin^{-1}\sqrt{y}\\
\sqrt{x_r}=\sin^{-1}\sqrt{1-x}&\quad \sqrt{y_r}=\sin^{-1}\sqrt{1-y};
\end{split}
\end{equation}
differentiating we see that
\begin{equation}
  \frac {dy_l}{dy}=\frac{\sin^{-1}\sqrt{y}}{\sqrt{y(1-y)}}\quad
\frac {dy_r}{dy}=-\frac{\sin^{-1}\sqrt{1-y}}{\sqrt{y(1-y)}}.
\end{equation}
 Substituting, we see that, for $b_0$ and $b_1$ non-zero, the leading term in the
solution operator $\hq_t$ is given by:
\begin{multline*}
q_{t,0}(x,y)= \varphi_0(x)  y_\ell^{-1} \left( \frac{y_\ell}{t}\right)^{b_0} e^{-\frac{x_\ell + y_\ell}{t}}
\psi_{b_0}(x_\ell y_\ell/t^2)\varphi_\ell(y)\frac{\sin^{-1}\sqrt{y}}{\sqrt{y(1-y)}} + \\
 (1-\varphi_0(x)) y_r^{-1} \left( \frac{y_r}{t}\right)^{b_1} e^{-\frac{x_r + y_r}{t}}
\psi_{b_1}(x_r y_r/t^2) \varphi_r(y)\frac{\sin^{-1}\sqrt{1-y}}{\sqrt{y(1-y)}}.
\end{multline*}

\section{The infinitesimal generator and long time asymptotics of solutions}
\label{s.infgen}
By Theorem~\ref{bscexstc}, $\cQ_t$ defines a semi-group on $\cC^0$; indeed, it
also defines a semi-group on $\cC^m$ for every nonnegative integer $m$. To
understand this $\cC^0$ semi-group better, and in particular to estimate the
long-time asymptotics of solutions, we now seek a characterization of its
infinitesimal generator $A$ as an unbounded operator on $\cC^0$, including some
features of its spectrum and a description of the behavior at $x=0$ and $x=1$
of the elements in its domain.  These facts will be proved using the various
regularity results we have obtained.  At various points in this discussion it
will be necessary bring in the the adjoint operator $L^t$, and in particular
the infinitesimal generator $A^*$ for the adjoint semi-group. Note that $A^*$
is an unbounded operator on $[\cC^0([0,1])]'$, which we identify with
$\cM([0,1]),$ the space of finite Borel measures on $[0,1].$
Section~\ref{adjsmgrp} contains a more complete discussion of the adjoint
semi-group.

The first step is to note that the infinitesimal generator has a compact resolvent.
\begin{proposition} Let $A$ be the infinitesimal generator associated to the $\cC^0$ 
semi-group  defined by $\cQ_t$. Then as an unbounded operator on $\cC^0([0,1]),$ 
the spectrum of $A$ lies in the left half-plane $\{\lambda:\:\Re\lambda\leq 0\}$; 
furthermore, $A$ has a compact resolvent.
\end{proposition}
\begin{proof} The first statement follows from the maximum principle, which
  implies that $\|\cQ_tf\|_{\cC^0}\leq \|f\|_{\cC^0}.$ Next, observe that for
  any $t>0$ and $f \in \cC^0([0,1])$, $\cQ_tf\in\CI([0,1])$ The closed graph
  theorem and Arzela-Ascoli theorem now apply to show that $\cQ_t$ is a compact
  operator on $\cC^0$ for any $t > 0$, so the second statement follows from the
  results in Section 8.2 of~\cite{DaviesLOS}.
\end{proof}

Of course, the full characterization of $A$ involves a detailed description of its domain. This will be
based on a basic result from semi-group theory, due to Nelson:
\begin{proposition}[Nelson,~\cite{nelson59}] Let $B$ be a Banach space and $Z$ a closed operator on 
$B$ generating a semi-group $T_t$. If $D\subset\Dom(Z)$ is a subspace which is dense in $B,$ and if 
$T_t D\subset D,$ for every $t>0,$ then $D$ is a core for $Z$, i.e.\ 
\[
Z=\overline{Z\restrictedto_D}.
\]
\end{proposition}

To apply this, recall that $\cQ_t\cC^2([0,1])\subset\cC^2([0,1])$ for $t>0,$ and also, if 
$f\in\cC^2([0,1])$, then $u=\cQ_tf$ satisfies $\pa_tu=Lu$ for $t \geq 0$. 
\begin{proposition}\label{p.infgen} If $A$ is the generator of the  $\cC^0$ semi-group defined 
by $\cQ_t,$ then
\[
A=\overline{L\restrictedto_{\cC^2([0,1])}}.
\]
\end{proposition}
\begin{proof} If $f\in\cC^2([0,1]),$ then $u(x,t)=\cQ_tf(x)$ has one time derivative and two 
spatial derivatives, all of which are continuous on $[0,1]\times [0,\infty)$. Now integrate the
equation satisfied by $u$ to compute that
\[
\frac{\cQ_tf-f}{t}=\frac{1}{t}\int\limits_0^tLu(x,s)ds.
\]
Since $Lu \in \cC^0([0,1]\times [0,\infty))$ and equals $Lf(x)$ at $t=0$, we have
\[
\frac{\cQ_tf-f}{t}=Lf+o(1).
\]
This implies that $\cC^2([0,1])\subset\Dom(A)$ and on this subspace, $Af=Lf.$ 
Since $\cC^2([0,1])$ is dense is $\cC^0([0,1])$, the proposition now follows directly
from Nelson's theorem.
\end{proof}
On the other hand, if $f\in\Dom(A)$ then $Lu\in\cC^0([0,1])$, so the one-dimensional version 
of ``elliptic regularity'' shows that $f\in\cC^2((0,1))$. In other words, the final characterization of
$\Dom(A)$ involves only the description of its elements at the boundaries. 

\subsection*{The case where neither $b(0)$ nor $b(1)$ vanish}
It turns out that the results in case either $b(0)$ or $b(1)$ vanish are slightly more complicated to state, so
for the moment let us suppose that $0<b(0), -b(1)$. As before, denote $b(0) = b_0$ and $-b(1) = b_1$. 

We begin by noting that there is a solution $v$ to the adjoint equation $L^tv=0$, where $L^t$ is given
in \eqref{eq:adjoint}, satisfying the adjoint boundary conditions \eqref{eq:adjbc}. Thus 
\begin{equation}\label{eqmeas}
v_0(x)=x^{b_0-1}(1-x)^{b_1-1}e^{B(x)},
\end{equation}
where $B(x)\in \calC^\infty([0,1])$, and 
\[
\pa_x[x(1-x)v_0](x)-b(x)v_0(x)=0.
\]
The existence of $v_0$ follows using standard ODE techniques. 

Choose $\varphi\in\CI([0,1])$ with support in $[0,1)$ such that $\varphi(x)=1,$ for $x$ in $[0,\frac12]$. 
If $f\in\cC^2$, then 
\begin{equation}\label{bdryidnt1}
\int\limits_0^1 (Lf(x)) \varphi(x)v_0(x)\, dx=\int\limits_{0}^1f(x) L^t(\varphi v_0)(x) \, dx.
\end{equation}
Since $\cC^2$ is dense in $\Dom(A),$ this identity also holds for the graph closure.  If $f\in\Dom(A),$ and $\delta$ is any
small positive number, then 
\[
\int\limits_{\delta}^1[Lf(x)\varphi(x) v_0(x)-f(x)L^t(\varphi v_0)(x)]\, dx= -\delta(1-\delta)v_0(\delta)\pa_x f(\delta).
\]
Using \eqref{bdryidnt1} and the asymptotic form of $v_0$, we deduce that 
\begin{equation}\label{bc0A}
\lim_{x\to 0^+}x^{b_0}\pa_x f(x)=0.
\end{equation}
A similar argument using a cutoff function with support near to $1$ shows that
\begin{equation}\label{bc1A}
\lim_{x\to 1^-}(1-x)^{b_1}\pa_x f(x)=0.
\end{equation}
Note that these are precisely the boundary conditions described in Section~\ref{natbvp}. They are also the ones described
by Feller as defining a positivity preserving contraction semi-group on $\cC^0([0,1]).$ 

If $f\in\cC^0_c((0,1)),$ then the definition of $\hq_t$ and the maximum principle imply
\[
\lim_{t\to 0^+}\int\limits_{0}^{1}\hq_t(x,y)f(y)\, dy=f(x)\text{ and }
\lim_{t\to 0^+}\int\limits_{0}^{1}\hq_t(x,y)f(x)\, dx=f(y).
\]
Using this and the sharp maximum principle for parabolic operators in one
dimension, see e.g.\ Theorem 2 in Chapter 3 of~\cite{ProtterWeinberger}, by a
straightforward limiting argument we obtain a strict pointwise lower bound for
$\hq_t:$
\begin{proposition} If $b(0),$ and $b(1)$ are non-vanishing, then for $t>0$ and $x,y\in [0,1],$ 
\begin{equation}\label{kerlb1}
\hq_t(x,y)>0.
\end{equation}
\end{proposition}
In this case, for each $t>0,$ the kernel $\hq_t(x,y)$ defines a strictly
positive operator on $\cC^0([0,1])$. In other words, if $f\in\cC^0([0,1])$ is
non-negative and not identically zero, then $\cQ_t f(x)>0$ for $x\in [0,1]$ and
$t > 0$.  Consequently, we can now apply Theorem 23.1 from~\cite{Lax} (The
Perron-Frobenius Theorem) to conclude that:
\begin{theorem}\label{spcthrybnz} If $b(0)$ and $b(1)$ are non-zero, then the infinitesimal generator $A$ is a 
  compact operator on $\cC^0([0,1])$ with spectrum lying in $\{\lambda:
  \Re\lambda \leq 0\}$. The only element in $\sigma(A)$ on the imaginary axis
  is the point $\lambda=0$, and the only associated eigenfunctions are the
  constant functions. The function $v_0,$ defined in~\eqref{eqmeas} spans the
  $0$-eigenspace of $A^*.$
\end{theorem} 

One consequence of the theorem above is that 
\[
\lambda_1=\sup\{\Re(\lambda):\: \lambda\in\sigma(A)\setminus\{0\}\} < 0.
\]
Thus defining
\[
u(x,t)=\int\limits_0^1\hq_t(x,y)f(y)\, dy, \text{ and } c_0=\int\limits_0^1v_0(y)f(y)\, dy,
\]
where $v_0$ is normalized to have integral $1,$ then Theorem 2.1 in Section B-IV of ~\cite{1prmsmgrpspops} 
implies the
\begin{corollary} Under the hypotheses of Theorem~\ref{spcthrybnz}, for each $\delta \in (0, |\lambda_1|)$, 
there is a constant $M_{\delta}>0$ such that
\[
\|u(x,t)-c_0\|_{\cC^0}\leq M_{\delta}\|f\|_{\cC^0}e^{-\delta t}.
\]
\end{corollary}

\subsection*{The case where either $b_0 = 0$ or $b_1 = 0$}
We now turn to the characterization of $\Dom(A)$, and the corresponding decay results for solutions,
when either $b_0$ or $b_1$ (or both) vanish. 

If $b_0=0,$ then by Proposition~\ref{p.infgen}, $Af(0)=0,$ for every $f\in\Dom(A).$  This in turn implies that 
$\delta(y)\in\Dom(A^*),$ and that $A^*\delta(y)=0.$ This is a new feature, since if $b_0 \neq 0$, then clearly
$\delta(y)\notin\Dom(A^*)$, and it is what complicates the discussion.  
Similar remarks apply at $x=1$ if $b_1=0$, of course.
 
More generally, if $d\mu\in\Dom(A^*)$ and $A^*d\mu=0,$ then by elliptic regularity in the open interval $(0,1)$, the 
measure $d\mu$ has the representation $g(y)dy$ where $g$ satisfies $L^t g = 0$ and the adjoint boundary 
conditions,~\eqref{eq:adjbc}, 
\begin{equation}\label{zeroflx2}
\pa_x[x(1-x)g(x)]-b(x)g(x)=0,
\end{equation}
so in particular $g \in \calC^\infty((0,1))$. 
Writing $b(x)=b_0(1-x)-b_1x+ x(1-x)\tb(x)$, then solutions to~\eqref{zeroflx2} have the form:
\begin{equation}\label{zeroflx3}
g(x)=C\frac{e^{-\tB(x)}}{x^{1-b_0}(1-x)^{1-b_1}},
\end{equation}
where $C$ is a constant and $\tB$ is a primitive of $\tb.$ Clearly, $g(x)dx$ is a measure of 
finite total variation if and only if both $b_0,\, b_1>0.$ 

If both $b_0 = b_1 =0,$ then $b(x)=x(1-x)\tb(x),$ where $\tb\in\CI([0,1]).$ The functions,
\begin{equation}\label{Anullsp00}
u_0(x)=C\int\limits_0^x\exp\left[-\int\limits_0^y\tb(z)dz\right]\,dy
\end{equation}
are strictly monotonically increasing in $(0,1),$ and solve $Lu=0.$  Choosing $C>0$ 
appropriately, we can assume that
\begin{equation}\label{Anullsp01}
u_0(0)=0\text{ and }u_0(1)=1.
\end{equation}
If $b_0 =0,$ but $b_1 > 0,$ then there is a non-negative solution $u_0,$ such that
\[
u_0(0) = 0, \qquad \text{and} \qquad \lim_{x\to 1^-}(1-x)^{b_1}\pa_xu_0(x)\neq 0,
\]
while f $b_0>0$, $b_1 = 0$, then there is a non-negative solution $u_0(x)$ with
\[
u_0(1)=0, \qquad \text{and} \qquad  \lim_{x\to 0^+}x^{b_0}\pa_xu_0(x)\neq 0.
\]

\medskip

If either $b_0$ or $b_1$ vanish, then the solution operator to~\eqref{genWFivp} can be expressed
in a form analogous to~\eqref{k0decmp}.   Suppose first that $b$ vanishes at only one endpoint, say
$b(0)=0,$ but $b(1)<0.$  We can use the kernel $k^{0,D}_t$ to build a solution operator, $\hq_t^{D}$,
for the Dirichlet problem at $x=0$ with the regular boundary condition at $x=1$: 
\[
\hq_t(x,y)=\hq_t^D(x,y)+\delta(y)c_0(x,t),
\]
where $(\pa_t-L)c_0(x,t)=0$. As before, $\hq_t^D(x,y)>0$ for $x>0$, and moreover,
\[
c_0(x,t)=1-\int\limits_0^1\hq_t^D(x,y)\,dy,
\]
so $c_0 \geq 0$.  A similar argument works if $b(0)>0,$ but $b(1)=0.$ Finally, if $b(0) = b(1) = 0$, then 
we can write
\begin{equation}\label{kersplt1}
\hq_t(x,y)=\hq_t^D(x,y)+\delta(y)c_0(x,t)+\delta(1-y)c_1(x,t),
\end{equation}
where $\hq^D_t(x,y)$ is positive on $(0,1)\times (0,1)$ and
\[
\begin{split}
c_0(x,t)+c_1(x,t)&=1-\int\limits_0^1\hq_t^D(x,y)\, dy,\\ c_1(x,t)&=1-\int\limits_0^1\hq_t^D(x,y)u_0(y)\, dy.
\end{split}
\]

In each of these cases, $\cC^0([0,1])$ splits into a finite dimensional
subspace, invariant under $\cQ_t,$ and an infinite dimensional complement, also
invariant under $\cQ_t.$ For example, if $b(0)=b(1)=0,$ then
\begin{equation}\label{Dsplit}
\cC^0([0,1])=\cC^0_0([0,1])\oplus\Span\{1,u_0\}.
\end{equation}
In all cases there is a corresponding splitting of the semi-group into the
Dirichlet semi-group, $\cQ^D_t,$ and a semi-group on a finite dimensional
space. The infinitesimal generator $A^D$ of $\cQ^D_t$ is the graph closure of
$L$ on the set of functions in $\cC^2([0,1])$ which vanish at the appropriate
end-point, or -points. The fact that $\hq^D_t$ is positive for $0<x,$ or $x<1,$ or
$0<x<1,$ respectively, implies as before that the semi-groups $\cQ^D_t$ are
positive and irreducible.

The next proposition follows from known results about positive, irreducible,
compact semi-groups acting on $\cC^0_0(X),$ where $X=(0,1),$ $(0,1],$ or
$[0,1),$ see~\cite{1prmsmgrpspops}.
\begin{proposition} If either or both of the numbers $b(0)$, $b(1)$ vanish,
  then $A^D$ is compact.  There is an element $\lambda_1\in\sigma(A^D)$ with
  $\lambda_1 \in (-\infty,0)$ and a unique corresponding eigenfunction $u_1$
  which is smooth and positive in $(0,1).$ The remainder of the spectrum lies
  in $\Re\lambda<\lambda_1-\eta,\,$ for some $\eta>0.$
\end{proposition}
\begin{proof} The compactness follows from the fact that $\cQ^D_t$ is compact
  for every $t>0.$ The existence of the eigenfunction $u_1$ and the negativity
  of $\lambda_1$ is obtained by applying oscillation and comparison theorems
  for Sturm-Liouville operators.  When $b(0)=b(1)=0$, comparison with operators
  of the form $M\pa_x^2+\frac{m\mu}{x(1-x)}$ yields the existence of
  $\lambda_1< 0$ and a unique associated eigenfunction $u_1 \in \Dom(A)$ with
  $u_1 > 0$ in $(0,1).$  The cases where only one of $b(0)$ or $b(1)$ vanish are somewhat
  easier. If $b(0)=0, b(1)<1,$ then the eigenvalue problem can written as
\[
\begin{split}
\pa_x[(1-x)^{b_1}e^{\tB(x)}\pa_xu]+\frac{\mu  e^{\tB(x)u(x)}}{x(1-x)^{1-b_1}}=0,\\
u(0)=0\text{ and }\lim_{x\to 1^-}(1-x)^{b_1}\pa_xu(x)=0.
\end{split}
\]
Since $(1-x)^{b_1-1}$ is integrable near $x=1$, we can apply standard
oscillation theorems to obtain the desired conclusion. The result then follows
from Corollary 2.2 in Section B-IV of ~\cite{1prmsmgrpspops} and the fact that
$A^D$ has a compact resolvent.
\end{proof}

If $b(0)=b(1)=0,$ then any $f \in \cC^0$ can be decomposed, according to ~\eqref{Dsplit}, as
\[
f=f_0+[f(0)+(f(1)-f(0))u_0],
\]
and then
\begin{equation}\label{grpsplt1}
\cQ_tf=\cQ^D_tf_0+[f(0)+(f(1)-f(0))u_0].
\end{equation}
The proposition above and Theorem 2.1, in Section B-IV of~\cite{1prmsmgrpspops}
imply that there is a constant, independent of $f$ so that
\[
\|\cQ^D_tf_0\|_{\cC^0}\leq Ce^{\lambda_1 t}\|f\|_{\cC^0}.
\]
Comparing the representations of $\cQ_tf$ in~\eqref{kersplt1}
and~\eqref{grpsplt1} shows that
\begin{equation}\label{atmasymp}
c_0(x,t)=1-u_0(x)+O(e^{\lambda_1t})\text{ and }  c_1(x,t)=u_0(x)+O(e^{\lambda_1t}).
\end{equation}
There are similar results for the other two cases, where only one of $b(0)$ or
$b(1)$ vanishes. In this case the solution tends asymptotically to the constant
$f(0),$ if $b(0)=0,$ and $f(1),$ if $b(1)=0.$

\section{The Resolvent of $A$}\label{s.resolv}
The Hille-Yosida theorem states that if $A$ is the infinitesimal
generator of a contraction semi-group $\cQ_t,$ then the right half plane
belongs to the resolvent set of $A.$ For $\lambda$ with positive real part the
resolvent, $(\lambda-A)^{-1},$ is given by the Laplace transform of $\cQ_t:$
\begin{equation}
  (\lambda-A)^{-1}=\int\limits_{0}^{\infty}e^{-\lambda t}\cQ_tdt.
\end{equation}
If $A$ is the $\cC^0$-graph closure of a generalized Wright-Fisher operator,
$L,$ then $\cQ_t,$ defined above, is a contraction semi-group on $\cC^0([0,1]),$
and therefore the right half-plane is in $\rho(A).$ In this section we consider
the higher order regularity of solutions to $(\lambda-A)w=f.$

Our regularity results show that if $f\in\cC^m([0,1])$ then the solution, $u$
to~\eqref{genWFivp} satisfies
\begin{equation}
  u\in\cC^0([0,\infty);\cC^m([0,1]))\cap \cC^{\infty}([0,1]\times (0,\infty)).
\end{equation}
We can therefore differentiate the equation satisfied by $u$ with respect to
$x$ to obtain that, for $1\leq j\leq m,$ and $t>0,$
\begin{equation}
  \pa_t[\pa_x^ju]=x(1-x)\pa_x^2[\pa_x^ju]+(b(x)+j(1-2x))\pa_x [\pa_x^ju]+c_j(x)[\pa_x^ju],
\end{equation}
for a function $c_j\in\CI([0,1]).$ The operator
\begin{equation}
  L_{[j]}=x(1-x)\pa_x^2+(b(x)+j(1-2x))\pa_x,
\end{equation}
is a generalized Wright-Fisher operator.  Applying a standard extension of the
maximum principle to equations with a zero order term, see Theorem 4 in Chapter
3.3 of~\cite{ProtterWeinberger}, we easily obtain
\begin{proposition}\label{highderest3}
  For $m\in\bbN,$ there are constants $C_m, \mu_m\geq 0,$ so that for $f\in
  \cC^m([0,1]),$ the solution to~\eqref{genWFivp} satisfies
  \begin{equation}\label{mdervest}
    \|u(\cdot,t)\|_{\cC^m([0,1])}\leq C_me^{\mu_m t}\|f\|_{\cC^m([0,1])}.
  \end{equation}
\end{proposition}
Combined with the continuity result in Theorem~\ref{bscexstc} this shows that
$\cQ_t$ defines a semi-group on $\cC^m([0,1])$ with $\|\cQ_t\|_{\cC^m}\leq
C_me^{\mu_m t}.$ We let the infinitesimal generators be denoted by $A_m.$  By
Nelson's theorem, these  can be taken as the $\cC^m$-graph-norm closure of $L$
acting on $\CI([0,1]).$ As $\cQ_t$ is a compact operator on $\cC^m([0,1])$ for
$t>0,$ the operators $A_m$ are compact for all $m\in\bbN.$

Evidently the resolvent set of $A_m$ contains the half-plane
$\{\Re\lambda>\mu_m\}.$ Because $A_m$ is a compact operator, if
$\lambda\in\sigma(A_m),$ then there is an eigenvector $u\in\Dom(A_m)$ so that
\begin{equation}
  Lu=\lambda u.
\end{equation}
Because $\cQ_t u=e^{\lambda t}u\in\CI([0,1]),$ we see that the eigenvectors all
belong to $\CI([0,1]),$ and therefore an eigenvector of $A_m$ is also an
eigenvector of $A,$ and vice versa. This shows that, as a point-set,
\begin{equation}\label{specm}
\sigma(A_m)=\sigma(A)\subset\{\lambda:\Re\lambda\leq 0\}.
\end{equation}
 
\begin{theorem} If $m\in\bbN,$ then for $f\in\cC^m([0,1])$ and $\lambda,$ with
  positive real part, the solution $w\in\Dom(A)$ to the equation
  \begin{equation}
    (\lambda-L)w=f,
  \end{equation}
belongs to $\cC^m([0,1])$ and if $\Re \lambda >\mu_m,$ then
\begin{equation}\label{lrglamest}
  \|w\|_{\cC^m}\leq \frac{\|f\|_{\cC^m}}{|\lambda-\mu_m|}.
\end{equation}
\end{theorem}
\begin{proof} The Hille-Yosida theorem shows that the estimate~\eqref{mdervest}
  implies that $\{\lambda:\Re\lambda>\mu_m\}$ belongs to the resolvent set of
    $A_m.$ For $f\in\cC^0$  we set
    \begin{equation}
      (\lambda-A)^{-1}f=\int\limits_0^{\infty}e^{-\lambda t}\cQ_tfdt.
    \end{equation}
    This is an analytic $\cC^0$-valued function in $\{\lambda:\:
    \Re\lambda>0\}.$ If $f\in\cC^m,$ then in $\{\lambda:\: \Re\lambda>\mu_m\},$
    this equals $(\lambda-A_m)^{-1}f,$ which is an analytic $\cC^m$-valued
    function in, $\rho(A_m),$ the resolvent set of $A_m.$ As noted
    in~\eqref{specm}, $\rho(A_m),$ includes $\{\lambda:\:
    \Re\lambda>0\},$ which completes the proof of the theorem.
\end{proof}

This is not, in any real sense, an elliptic estimate, as it only shows that the
solution $w$ is at least as regular as $f.$ Of
course in the interior of the interval, $w$ has two more derivatives than $f,$
but this may not be true, in a uniform sense, up to the boundary.

\section{The adjoint semi-group}\label{adjsmgrp}
It is a consequence of the boundary behavior of $\hq_t,$ which follows from
Corollary~\ref{kernest1}, that for $t>0,$ and each $x\in [0,1],$
\begin{equation}\label{frmadj}
L_y^t\hq_t(x,\cdot)\in L^1([0,1])
\end{equation}
The results in Section~\ref{htker4L} show that, if $f\in \cC^m([0,1])$ for $m \geq 0,$ then
\[
u(x,t)=\int\limits_0^1\hq_t(x,y)f(y) \, dy \in \cC^0([0,\infty)_t;\cC^m([0,1]_x))
\]
and satisfies $\pa_tu=Lu$ with $u(x,0)=f(x)$. Thus for $f\in\cC^2([0,1]),$
uniqueness for this initial value problem shows that, for $t>0,$
\[
Lu(x,t)=\int\limits_0^t\hq_t(x,y)Lf(y)dy.
\]
By Corollary~\ref{kernest1}, $\hq_t$ satisfies the adjoint boundary
conditions~\eqref{eq:adjbc}, so the integrability of $L_y^t\hq_t,$ implies that
we can integrate by parts to conclude that
\[
Lu(x,t)=\int\limits_0^tL_y^t\hq_t(x,y)f(y)\,dy.
\]
Since 
\[
\pa_t u=\int\limits_0^1\pa_t\hq_t(x,y)f(y)\, dy,
\]
it follows immediately that, for all $f\in\cC^2([0,1])$ and $t>0,$ 
\[
\int\limits_0^t(\pa_t-L_y^t)\hq_t(x,y)f(y)\, dy=0,
\]
and hence we conclude by a straightforward limiting argument the following:
\begin{theorem}\label{frdeqn}
If $g\in\cC^0([0,1]),$ then, for $t>0$
\[
v(y,t)=\int\limits_{0}^{1}\hq_t(x,y)g(x)\, dx,
\]
satisfies the boundary conditions~\eqref{eq:adjbc}, and solves the initial value problem
\begin{equation}\label{foreqn}
 (\pa_t-L^t)v(y,t)=0\text{ and }\lim_{t\to 0^+}v(y,t)=g(y).
\end{equation}
\end{theorem}

The dual space of $\cC^0([0,1])$ is naturally identified with $\cM([0,1])$, the
space of Borel measures with finite total variation on $[0,1].$ The dual
semi-group, $\cQ_t',$ is thus canonically defined on this space by
\[
\langle\cQ_tf,d\mu\rangle=  \langle f,\cQ_t'd\mu\rangle.
\]
However, $\cC^0([0,1])$ is \emph{not} a reflexive Banach space, so the dual semi-group
is weak${}^*$-continuous, but not necessarily strongly
continuous. The infinitesimal generator $A$ of $\cQ_t$ has a canonically
defined adjoint, $A^*,$ whose domain is defined by the  prescription:
a measure $d\nu\in\Dom(A^*)$ if there exists a constant $C$ so that for every
$f\in\Dom(A),$
\[
\left|\int\limits_{0}^{\infty}Af(x)d\nu(x)\right|\leq C\|f\|_{\cC^0}.
\]
The subtlety is that $\Dom(A^*)$ may not be dense in $\cM([0,1]).$ Following
Phillips, see~\cite{HillePhillips}, we define the adjoint semi-group as
\[
\cQ_t^{\odot}=\cQ'_t\restrictedto_{\cM^{\odot}},
\text{ where }\cM^{\odot}=\overline{\Dom(A^*)}\cap A^*\Dom(A^*).
\]
Phillips shows that $\cQ_t^{\odot}$ is a strongly continuous semi-group on
$\cM^{\odot},$ with infinitesimal generator:
\[
A^{\odot}=A^*\restrictedto_{\Dom(A^*)\cap cM^{\odot}}.
\]
Thus our task is to identify $\cM^{\odot}$.  

By Theorem~\ref{frdeqn}, measures of the form $d\mu=g(y)dy$ with
$g\in\cC^2((0,1))$ are in $\Dom(A^{\odot})$ provided $L^tg \in L^1$ and $g$
satisfies the adjoint boundary conditions~\eqref{eq:adjbc}.  The closure of all
such measures with respect to the topology of $\cM([0,1])$ is $L^1([0,1])$. Any
classical eigenvector of $L^t$ satisfying ~\eqref{eq:adjbc} belongs to
$\cM^{\odot}.$ Elliptic regularity implies that a distributional solution to
$L^t d\mu=d\nu$ is absolutely continuous on $(0,1)$, and hence
\[
L^1([0,1])\subset\cM^{\odot}.
\]
In fact, except for the possibility of atomic measures at $0$ and/or $1$, these
two spaces are equal.

If $b(0)\neq 0,$ then for any $f\in\cC^2([0,1])$, 
\[
\langle Lf,\delta(y)\rangle=b(0)\pa_xf(0).
\]
Since the right hand side does not represent a bounded functional on
$\cC^0([0,1]),$ we obtain that $\delta\notin\Dom(A^*).$ A similar calculation
can be done at $x=1.$ On the other hand, if $b$ vanishes at either end, then:
\begin{enumerate}
\item If $b(0)=b(1)=0,$ then the nullspace of $A^{\odot}$ is  spanned by
  $\delta(y)$ and $\delta(1-y)$, both of which belong to $\cM^{\odot}.$
\item If one of $b(0)$ or $b(1)$ is non-zero, then $A^{\odot}$ has a
  $1$-dimensional nullspace spanned by $\delta(y),$ if $b(0)=0,$ or
  $\delta(1-y),$ if $b(1)=0$, and this nullspace again lies in $\cM^{\odot}.$
\end{enumerate} 

\noindent
Summarizing these observations we have proved the 
\begin{proposition} 
\begin{itemize}
\item If $b(0)$ and $b(1)$ are non-vanishing, then $\cM^{\odot}=L^1([0,1]).$ 
\item If $b(0)=b(1)=0,$ then $\cM^{\odot}=L^1([0,1])\oplus\Span\{\delta(y),\delta(1-y)\};$ 
\item If $b(0)=0,$ $b(1)\neq 0,$ then $\cM^{\odot}=L^1([0,1])\oplus\Span\{\delta(y)\};$ 
\item if $b(0)\neq 0,$ $b(1)= 0,$ then $\cM^{\odot}=L^1([0,1])\oplus\Span\{\delta(1-y)\}.$
\end{itemize}
\end{proposition} 

The spectrum of the operator $A^{\odot}$ equals that of $A,$ see Theorem 14.3.3
in~\cite{HillePhillips}; therefore we have the same sort of asymptotics for
solutions to~\eqref{foreqn}:
\begin{theorem} If $0<b(0)$ and $b(1)<0,$ then there exists a $\lambda_1<0,$ so
  that the solution $v(x,t)$ to~\eqref{foreqn}, with initial data $f\in
  L^1([0,1])$ satisfies
\[
v(x,t)=c_0v_0(x)+O(e^{\lambda_1t}),\text{ where }c_0=\int\limits_0^1f(y)\, dy.
\]
Here $v_0\in\Dom(A^*)$ is the positive solution to $L^tv_0=0,$ normalized to
have integral $1$. If $b(0)=b(1)=0,$ then there is a $\lambda_1<0$ such that
\begin{multline*}
v(x,t)=c_0\delta(y)+c_1\delta(1-y) +O(e^{\lambda_1 t}),\\
\text{where } c_0=\int\limits_0^1(1-u_0(y))f(y), 
\quad \text{and} \quad c_1=\int\limits_0^1u_0(y)f(y)\, dy.
\end{multline*}
where $u_0$ is defined in~\eqref{Anullsp00} and~\eqref{Anullsp01}. Finally, if
only one of $b(0)$ or $b(1)$ vanishes, then
\[
v(x,t)= 
\begin{cases}
 c\delta(y)+O(e^{\lambda_1 t}), \qquad & b(0)=0, \\
c \delta(1-y)+O(e^{\lambda_1 t}), & b(1)=0,  
\end{cases}
\qquad  c=\int\limits_0^1u_0(y)f(y)\, dy.
\]
\end{theorem}
\begin{proof} The first case is immediate from the fact that $0$ is an isolated
  point in the spectrum of $A^{\odot},$ and all other eigenvalues have strictly
  negative real part. The second case follows from~\eqref{atmasymp} and
  the analogous fact about the spectrum of $A^{D *}$; the last two assertions
  are similar.
\end{proof}

As a simple special case of our regularity theorem for solutions of the backwards
equation we have:
\begin{corollary} If $u(x,t)$ is a solution to
  \begin{equation}
    \pa_t u=x(1-x)\pa_x^2 u\quad u(x,0)=f(x),
  \end{equation}
  for $f\in\cC^{2l}([0,1]),$ then, for $(j+k)\leq l$, the functions
  $\pa_t^jL^ku$ are continuous on $[0,1]\times [0,\infty).$
\end{corollary}
If $v$ solves the forward Kolmogorov equation:
\begin{equation}
  \pa_tv=\pa_y^2[y(1-y) v]\text{ with }v(y,0)=g(y)\in\cC^{2l}([0,1]),
\end{equation}
then it is a simple calculation to see that $u(x,t)=x(1-x)v(x,t),$ solves the
backwards equations, with $u(x,0)=x(1-x)g(x),$  and $u(0,t)=u(1,t)=0.$

The theorem shows that $x(1-x)v(x,t)$ is therefore a $\cC^{2l}$ function on
$[0,1]\times [0,\infty),$ vanishing at $0$ and $1.$ This easily implies that
$v$ itself is in $\cC^{2l}([0,1]\times [0,\infty)).$ 
If we let $L^t_{\WF}g=\pa_x^2x(1-x)g,$ then this regularity shows that
for any $1\leq j\leq l$
\begin{equation}\label{tderiv1}
  \pa_tL^{t\,(j-1)}_{\WF}v=L^{t\,j}_{\WF}v.
\end{equation}
We can integrate the forward equation, and use this formula, repeatedly
integrating by parts, to obtain
\begin{equation}
\begin{split}\label{4}
  v(x,t)&=g(x)+\int\limits_{0}^tL^t_{\WF}v(x,s)ds\\
&=g(x)+(s-t)L^t_{\WF}v(x,s)\Bigg|^{s=t}_{s=0}+\int\limits_{0}^t(t-s)L^{t\, 2}_{\WF}v(x,s)ds\\
&=\sum_{j=0}^{l-1}\frac{t^jL^{t\,j}_{\WF}g}{j!}+
\frac{1}{(l-1)!}\int\limits_{0}^t(t-s)^lL^{t\, l}_{\WF}v(x,s)ds.
\end{split}
\end{equation}

Using the same argument, we can show that if $f\in\cC^{2l}([0,1])$ and $u$
solves the initial value problem: $(\pa_t-L)u=0,\, u(x,0)=f(x),$ for $L$ a
generalized Wright-Fisher operator, then
\begin{equation}
  u(x,t)
=\sum_{j=0}^{l-1}\frac{t^jL^{j}f}{j!}+
\frac{1}{(l-1)!}\int\limits_{0}^t(t-s)^lL^{l}u(x,s)ds.
\end{equation}
Applying Proposition~\ref{highderest3} show that there is a constant $M_{2l}$ so that
\begin{equation}
  |L^lu(x,s)|\leq M_{2l}e^{\mu_{2l}s}\|f\|_{\cC^{2l}},
\end{equation}
and therefore:
\begin{equation}
  \left| u(x,t)
-\sum_{j=0}^{l-1}\frac{t^jL^{j}f}{j!}\right|\leq
\frac{M_{2l}\|f\|_{\cC^{2l}}e^{\mu_{2l}t}t^l}{l!}.
\end{equation}

\appendix
\section{Appendix}\label{pertthryprfs}

\noindent{{\it Proof of Proposition \ref{prop1515}}: }

Consider the functions
\[
v(x,t;s)=\int\limits_{0}^{\infty}k_{t-s}^b(x,y)yh(y)\pa_y g(y,s)\, dy,
\]
defined when $t \geq s$; these satisfy
\[
(\pa_t-L_b)v=0 \text{ and } v(x,s;s)=xh(x)\pa_x g(x,s).
\]
By the maximum principle and the derivative bound on $g$ it follows that
\[
|v(x,t;s)|\leq \sup |xh(x)M/\sqrt{sx}| \leq \sqrt{L}M\|h\|_{\infty},
\]
and hence
\[
|A^b_tg(x)|\leq
\int\limits_0^tM\sqrt{L}\|h\|_{\infty}\frac{ds}{\sqrt{s}}=2M\sqrt{Lt}\|h\|_{\infty}. 
\]
This establishes~\eqref{Aestj} for $j=1$, with $d_0=1$. Now apply Lemma~\ref{lem2} to get
\[
|\pa_xv(x,t;s)|\leq\frac{C_b\|v(\cdot,s;s)\|_{\infty}}{\sqrt{x(t-s)}},
\]
so that
\[
\begin{split}
|\pa_xA^b_tg(x)|&\leq\int\limits_0^t\frac{C_bM\sqrt{L}\|h\|_{\infty}ds}{\sqrt{xs(t-s)}}\\
&=\frac{C_b\pi M\sqrt{L}\|h\|_{\infty}}{\sqrt{x}};
\end{split}
\]
setting $d_1 = \pi$, this is \eqref{dAestj} for $j=1$. 

Assume that we have chosen $d_0, \ldots, d_{j-1}$. Now write
\[
v_j(x,t;s) = \int_0^\infty k_{t-s}^b(x,y) y\pa_y v_{j-1}(y,s)\, dy.
\]
Using \eqref{dAestj} for $j-1$, this is bounded by
\begin{multline*}
d_{j-1}M C_b^{j-1} (\sqrt{L} ||h||_\infty)^{j} \int_0^\infty k_{t-s}^b(x,y) s^{\frac{j-1}2} \, dy \\ 
= \frac{2}{j} d_{j-1}M C_b^{j-1}  (\sqrt{L} ||h||_\infty)^{j} t^{\frac{j}{2}}, 
\end{multline*}
since $\int_0^\infty k_{t-s}^b(x,y)\, dy = 1$ for $s < t$, which establishes~\eqref{Aestj} for $j$.

Using the induction hypothesis one more time, insert the result into the estimate of Lemma~\ref{lem2}.
Noting that $1/( t-s + \sqrt{x (t-s)}) \leq 1/\sqrt{x (t-s)}$, we get
\[
\begin{split}
|\pa_x ( A^b_t ) ^jg(x)|&\leq \int\limits_0^t
\frac{d_{j-1}M(C_b\sqrt{L}\|h\|_{\infty})^js^{\frac{j-2}{2}}ds}{\sqrt{x(t-s)}}\\
&=\frac{d_{j-1}M(C_b\sqrt{L}\|h\|_{\infty})^jt^{\frac {j-1}2}}{\sqrt{x}}
\int\limits_0^1\frac{\sigma^{\frac{j-2}{2}}d\sigma}{\sqrt{1-\sigma}}.
\end{split}
\]
Evaluating the integral shows that we should set
\[
d_j=\sqrt{\pi} d_{j-1}\frac{\Gamma\left(\frac{j}{2}\right)}{\Gamma\left(\frac{j+1}2 \right)},
\]
and this proves~\eqref{dAestj} for $j$, thereby completing the induction.  A straightforward calculation 
shows that $d_j$ is given by the formula in the statement of the Proposition.
\hfill $\Box$

\bigskip

\noindent{{\it Proof of Proposition \ref{atbests2}}:}

We start with $j=1.$  These results follow from~\eqref{leibfrm0},  the maximum principle, and 
Proposition~\ref{prop1515}. Combining these ingredients shows that when $t<T,$
\[
|\pa^\ell_xA^b_tg(x)|\leq  \sqrt{t}[2^{\ell+1}\|xh\|_{\cC^\ell}](M+\|g\|_{\cC^{\ell,\infty}([0,T])}),
\]
and,
\begin{multline*}
|\pa^{\ell+1}_xA^b_tg(x)|\leq\\
\frac{d_1MC_{b+\ell}\sqrt{L}\|h\|_{\infty}}{\sqrt{x}}+\sqrt{t}[2^{\ell+1}\|xh\|_{\cC^{\ell+1}}(1+\sqrt{L})](\|g\|_{\cC^{\ell,\infty}}+M),
\end{multline*}
which establishes \eqref{Aestjl} and~\eqref{dAestjl} when j=1. The main issue is to see how $D_j$ 
decreases as $j$ increases. We assume that these estimates have been established for $\{1,\dots, j-1\}.$

Applying~\eqref{leibfrm0}, we see that
\begin{multline*}
\pa_x^l(A^b_t)^jg(x)=A_t^{b+\ell}(\pa_x^\ell(A^b_{\cdot})^{j-1}g)(x)+\\ \sum_{p=1}^\ell
\left(\begin{matrix}\ell \\p\end{matrix}\right)
\int\limits_{0}^{t}\int\limits_0^{\infty}k^{b+\ell}_{t-s}(x,z)\pa_z^{p}(zh(z))\pa_z^{\ell+1-p} (A^b_s)^{j-1}g)(z)\, dzds.
\end{multline*}
Using the induction hypothesis and the maximum principle, the first term on the right here is bounded by
\[
\frac{2D_{j-1}\sqrt{L}\|h\|_{\infty}t^{\frac{j}{2}}}{j}(M+\|g\|_{\cC^{\ell,\infty}}) [C'_{T,L,\ell,b}\|h\|_{\cC^{\ell}}]^{j-1}.
\]
The terms in the sum with $2\leq p\leq \ell$ can be bounded using the maximum principle and Lemma~\ref{higdest}, while 
the term with $p=1$ is controlled using the induction hypothesis. Thus altogether, the sum is bounded by
\[
\frac{(2^\ell \|xh\|_{\cC^\ell})^j\|g\|_{\cC^\ell}t^j}{j!}+\frac{2D_{j-2}t^{j+1}{2}(M+\|g\|_{\cC^{\ell,\infty}})\|xh\|_{\infty}[C'_{T,L,\ell,b}
\|h\|_{\cC^\ell}]^{j-1}}{j(j-2)}.
\]
So long as 
\begin{equation}\label{dlwrbnd}
D_{j-1}\geq\max\left\{\frac{1}{j!},\frac{2D_{j-2}}{j-1}\right\},
\end{equation}
then there is some $C'_{T,L,\ell,b}$ so that~\eqref{Aestjl} holds for all $j.$ That~\eqref{dlwrbnd} 
holds follows easily from the proof of~\eqref{dAestjl}, to which we now turn.

We next apply~\eqref{leibfrm0} to see that
\begin{multline*}
\pa_x^{\ell +1}(A^b_t)^jg(x)=\pa_x\int\limits_0^t\int\limits_0^{\infty} k^{b+\ell}_{t-s}(x,z)zh(z)\pa_z^{\ell+1}(A^b_s)^{j-1}g)(z)\, dzds+\\
\left(\begin{matrix}\ell \\1\end{matrix}\right)\pa_x\int\limits_0^t\int\limits_0^{\infty}k^{b+\ell}_{t-s}(x,z)
\pa_y(zh(z))\pa_z^{\ell}(A^b_s)^{j-1}g)(z)\, dzds+\\ \sum_{p=2}^\ell \left(\begin{matrix}\ell \\p\end{matrix}\right)
\int\limits_{0}^{t}\int\limits_0^{\infty}k^{b+\ell +1}_{t-s}(x,z)\pa_z\left[\pa_z^{p}(zh(z))\pa_z^{\ell +1-p} 
(A^b_s)^{j-1}g)(z)\right]\, dzds.
\end{multline*}
By the induction hypothesis and Lemma~\ref{lem2}, the first term is estimated by
\[
\begin{split}
&\frac{D_{j-1}(M+\|g\|_{\cC^{\ell,\infty}})(C_{T,L,\ell,b}\|h\|_{\cC^{\ell+1}})^{j-1}}{\sqrt{x}}
\int\limits_{0}^tC_{b+\ell}\sqrt{L}\|h\|_{\infty}\frac{s^{\frac{j-2}{2}}ds}{\sqrt{t-s}}\\
& = \sqrt{\pi}D_{j-1}t^{\frac{j-1}{2}}\frac{\Gamma(\frac{j}{2})}{\Gamma(\frac{j+1}{2})}
\frac{(M+\|g\|_{\cC^{\ell,\infty}})(C_{T,L,\ell,b}\|h\|_{\cC^{\ell+1}})^{j-1}}{\sqrt{x}},
\end{split}
\]
and the second term by:
\[
\frac{\sqrt{\pi}2D_{j-2}t^{\frac{j}{2}}}{j-2}\frac{\Gamma(\frac{j}{2})}{\Gamma(\frac{j+1}{2})}
\frac{(M+\|g\|_{\cC^{\ell,\infty}})(C_{T,L,\ell,b}\|h\|_{\cC^{\ell+1}})^{j-1}}{\sqrt{x}}.
\]
Using the induction hypothesis, Lemma~\ref{higdest} and the maximum principle,
the last term is bounded by
\[
2^\ell \frac{\sqrt{\pi}4D_{j-2}t^{\frac{j}{2}}}{j(j-2)}(M+\|g\|_{\cC^{\ell,\infty}})(C'_{T,L,\ell,b}\|h\|_{\cC^{\ell+1}})^{j}.
\]
Hence if  $D_j$ satisfies the recursion relationship in the statement of the proposition, 
then $D_{j-1}$ satisfies~\eqref{dlwrbnd}, for $j\geq 2$ and there exist constants $C_{T,L,\ell,b}$ and $C'_{T,L,\ell,b}$ 
so that~\eqref{Aestjl} and~\eqref{dAestjl} hold for all $j\geq 2.$ \hfill $\Box$

{\bibliographystyle{siam} {\bibliography{alla-k}}}

\end{document}